\documentclass[twoside,reqno]{amsart}
\usepackage{times}
\usepackage[margin=1in]{geometry}

\title[On Unique Ergodicity in Stochastic PDEs]{\Large On Unique Ergodicity in Nonlinear Stochastic Partial Differential Equations}
\author[Nathan Glatt-Holtz, Jonathan C. Mattingly, Geordie Richards]{Nathan Glatt-Holtz, Jonathan C. Mattingly, Geordie Richards\\
  \scriptsize{emails: negh@math.vt.edu, jonm@math.duke.edu,
    g.richards@rochester.edu}}
\dedicatory{Dedicated to David Ruelle and Yakov G. Sinai on the
  occasion of their 80th birthdays with thanks for all they have and will
teach us.}
\date{}

\usepackage{amsfonts, amssymb, amsmath, amsthm}
\usepackage[colorlinks=true, pdfstartview=FitV, linkcolor=blue,
            citecolor=blue, urlcolor=blue]{hyperref}

\numberwithin{equation}{section}

\newtheorem{theorem}{Theorem}[section]
\newtheorem{lemma}{Lemma}[section]
\newtheorem{proposition}{Proposition}[section]
\newtheorem{corollary}{Corollary}[section]

\newtheorem{remark}{Remark}[section]

\usepackage{mathtools}
\mathtoolsset{showonlyrefs}

\renewcommand{\H}{H}
\newcommand{\pd}[1]{\partial_{#1}}
\newcommand{\indFn}[1]{1 \! \! 1_{#1}}
\newcommand{\one}[1]{\indFn{#1}}
\newcommand{\E}{\mathbb{E}}
\newcommand{\Prb}{\mathbb{P}}
\newcommand{\TT}{\mathbb{T}}
\newcommand{\DD}{\mathcal{D}}

\newcommand{\spH}{H}
\newcommand{\spV}{V}

\newcommand{\RR}{\mathbb{R}}

\newcommand{\N}{\mathbf{N}}

\newcommand{\bfU}{\mathbf{u}}
\newcommand{\tbfU}{\tilde{\mathbf{u}}}
\newcommand{\bfUe}{\mathbf{u}^\epsilon}
\newcommand{\bfV}{\mathbf{v}}
\newcommand{\pres}{\pi}
\newcommand{\tpres}{\tilde{\pi}}
\newcommand{\dpres}{q}
\newcommand{\qe}{\xi^\epsilon}
\newcommand{\Vort}{\xi}

\newcommand{\vshft}{r}
\newcommand{\yshft}{y}

\newcommand{\wdiag}{w}

\newcommand{\spcdot}{\,\cdot\,}

\begin{document}

\begingroup
\def\uppercasenonmath#1{} 
\let\MakeUppercase\relax 
\maketitle
\endgroup

\begin{abstract}
  We illustrate how the notion of asymptotic coupling provides a
  flexible and intuitive framework for proving the
  uniqueness of invariant measures for a variety of stochastic partial differential equations whose deterministic
  counterpart possesses a finite number of determining modes.   Examples exhibiting parabolic and hyperbolic structure
  are studied in detail.  In the later situation we also present a simple framework for establishing the existence of invariant measures
  when the usual approach relying on the Krylov-Bogolyubov procedure and compactness fails.
\end{abstract}

\section{Introduction}

The goal of this work is to give a simple exposition, distillation and
refinement of methods developed over the last decade and a half
to analyze ergodicity in nonlinear stochastic PDEs with an additive
forcing.
To this end, we detail a number of examples which highlight different
difficulties and help clarify the domain of applicability and
flexibility of the core ideas.
For each example, we provide a simple proof of unique
ergodicity with a presentation which should be adaptable to other settings.
Though our calculations often lay the ground work for stronger results
such as
convergence of transition measures, exponential mixing or
spectral gaps, we resist the urge to expand the discussions here, and opt to make the uniqueness
arguments as simple as possible.  Although some of our examples are close
to those in the existing literature, many are not and require an involved
analysis to develop the required PDE estimates.  For all the situations
considered we present a relatively succinct proof of unique ergodicity,
particularly when compared with existing expositions.

The feature common to all of our examples is the existence of a finite
number of determining modes in the spirit of  \cite{FoaisProdi67} and a sufficiently
rich stochastic forcing structure to ensure that all determining modes
are directly excited.
There has been a larger body of work in these directions in recent years
beginning with \cite{BricmontKupiainenLefevere2001, EMattinglySinai2001,KuksinShirikyan1}
and continuing with \cite{Mattingly2002,KuksinShirikyan2, Mattingly2003,Hairer02,ELiu2002,
HairerMattingly06,HairerMattingly2008,DebusscheOdasso2008,HairerMattingly2011,
HairerMattinglyScheutzow2011, KuksinShirikian12,FoldesGlattHoltzRichardsThomann2013} to name a few.
Very roughly speaking, the presence of noise terms allows for the `coupling' of all
the relevant large scales of motion, which contain any unstable
directions.  The small scales, which are then provably stable,
contract asymptotically in time.

The heart of the calculations presented below are very much in the
spirit of \cite{EMattinglySinai2001,Mattingly2003,BakhtinMattingly2005} although
the approach here does not pass through a reduction to an equation
with memory. In that sense our presentation is closer to
\cite{Mattingly2002,Mattingly2003} which decomposes the future
starting from an initial condition and proves a coupling along a
subset of futures of positive probability.  There however, the
analysis was complicated by an attempt at generality and the
desire to prove exponential convergence. In \cite{Hairer02}, which
followed \cite{Mattingly2002}, the control used to produce the
coupling drove all of the modes together only asymptotically.\footnote{\cite{Hairer02} also coined the term `asymptotic
 coupling' which was later defined more generally in \cite{HairerMattingly2011}.} In
particular, \cite{Hairer02} did
 not force the large scales (``low modes'') to match exactly as was the case in
previous works. While this leads to slightly weaker results, it
can be conceptually simpler in some settings. In parallel to
these works,  two other groups developed their own
takes on these same questions. One vein of work is contained in
\cite{BricmontKupiainenLefevere2001,BricmontKupiainenLefevere02} and
the other beginning in \cite{KuksinShirikyan1,KuksinShirikyan2}
is nicely summarized in \cite{KuksinShirikian12}.

We proceed through the lens of a variation on the `asymptotic coupling' framework from
\cite{HairerMattinglyScheutzow2011} (and equally in the spirit of \cite{EMattinglySinai2001, Hairer02,Mattingly2003}).
This formalism allows us to highlight the underlying flexibility and the wide range applicability
of the above mentioned body of work by treating a
number of interesting systems simply and without extraneous
complications dictated by previous abstract frameworks.
Indeed, the examples selected below were chosen to underline a variety of commonly encountered
difficulties which can be surmounted, including the lack of exponential moments of
critical norms, the lack of well-posedness in the space where the
convergence analysis is performed, or, as in the case of weakly damped hyperbolic
systems, situations in which the dissipative mechanism is uniform across
(spatial) scales.

Our first example is the most classical: the 2-D Navier-Stokes
equations (NDEs) posed on a domain.  Here the presence of boundaries prevents
a closed vorticity-formulation and hence interferes with higher order
constraints of motion resulting in a `critical' problem.  This
criticality makes attaining the gradient bounds on the
Markov-semigroup difficult. Such bounds are central to the
infinitesimal approach of Asymptotic Strong Feller developed in
\cite{HairerMattingly06,HairerMattingly2011} to address the
hypoelliptic setting. As currently presented, the Asymptotic Strong Feller approach
does not localize easily. However the analysis presented in
\cite{EMattinglySinai2001,Mattingly2002,Mattingly2003} was localized
from the start, though the estimates used there do not apply to this
setting directly due to the lack of a vorticity formulation. Our presentation
is simplified in comparison to previous works, producing an
immediate and transparent proof of unique ergodicity.

We then turn to address two other interesting dissipative equations
arising from fluid systems which have received much less attention in
the SPDE literature.  The first example is provided by the so-called
hydrostatic Navier-Stokes (or simplified
Primitive Equations) arising for fluids spanning geophysical scales
and therefore of interest in climate and weather applications. See
\cite{PetcuTemamZiane2008} and Section~\ref{sec:SPE} for extensive
further references.  Our second example provides a streamlined analysis of the
so-called fractionally dissipative
stochastic Euler equation, introduced recently in
\cite{ConstantinGlattHoltzVicol2013}.  The theme shared by these
examples is that the non-linearity is relatively stronger than the
dissipative structure in comparison to the 2D NSEs.  This leads to a
situation in which the continuous dependence of solutions, and hence
the Foias-Prodi estimate, is tractable only when carried out in a
weaker topology.  As made explicit in Corollary~\ref{cor:precise:app} (see below),
performing this convergence analysis in a weaker topology is
sufficient to provide a suitable asymptotic coupling.

The last two examples are illustrative of the difficulties encountered in studying weakly damped
hyperbolic systems.  The first equation is a variation of the damped Euler-Voigt equation,
a hyperbolic regularization of the Euler equations.   As a second example we address the weakly damped stochastic
Sine-Gordon equations.     Here rather than using the parabolic
structure to produce an effective large damping at small scales,
higher-order regularity constraints restrict the strength of transfer to high frequencies and allow us to obtain
stronger control on the non-linear terms.

Regarding the Euler-Voigt equation, it is notable that the existence of solutions
also requires special consideration.    For this stage in the analysis, the
lack of any obvious finite time smoothing mechanism or alternatively of any invariant quantities in spaces more regular
than those for which the equations are well posed suggest that the usual approach relying on the
Krylov-Bogolyubov procedure and compactness may fail.  To address this difficult we show how a limiting procedure involving
a parabolic regularization can be used to guarantee the existence of stationary states.    This stage of the analysis
makes use of another abstract criteria which we think will prove useful in other future application in hyperbolic
SPDEs.

Note that in all of these examples we will focus our attention exclusively on the ``effectively elliptic''
setting where all of the ``determining modes'', or ``presumptively
unstable directions'', are directly forced stochastically.  The
hypoelliptic setting, where most of the determining modes are not
directly forced and the drift is used to spread the noise to all of
the determining modes, remains unexplored for all of the examples we
discuss and will almost certainly require significantly more
machinery, though much is provided by
\cite{HairerMattingly06,HairerMattingly2011}.
We emphasize that while we may extend our analysis in many of the examples
to obtain convergence rates or possibly spectral gaps,
we only explicitly address the question of unique ergodicity here in order to keep the exposition minimal and
to maintain the broadest range of applicability. That being said, the
analysis here lays the framework for obtaining convergence rates using
the ideas outlined in \cite{Mattingly2002,Hairer02,Mattingly2003}.

The manuscript is organized as follows. In Section~\ref{sec:abs:frame}, we recall the asymptotic coupling
framework and introduce several refinements which were not made explicit in previous work.  We conclude
this section by recalling a specific form the Girsanov theorem which will be crucially used to apply the
abstract asymptotic coupling in each of the examples and give a
general recipe for building asymptotic couplings.  In Section~\ref{sec:examples}, we successively
treat each of the SPDE examples described, proving the existence and uniqueness of an ergodic invariant measure
in each case.   The appendix is devoted to proving some abstract lemmata which
we use to prove the existence of an invariant measure for Euler-Voigt system consider in Section~\ref{sec:e:voigt}.

\section{An Abstract framework for Unique ergodicity}
\label{sec:abs:frame}

We now introduce the simple abstract framework which we use for proving unique ergodicity.
After briefly recalling some generalities about Markov processes and ergodic theory
we  present a refinement of the asymptotic coupling arguments developed
in \cite{HairerMattinglyScheutzow2011}.
The approach is very much in the spirit of the general results given in
this previous work but the packaging here is a
little different. We emphasize the exact formulation we will use and
make explicit the ability to establish convergence in a different topology than the one
associated to the space on which the Markov dynamics are
defined (see Corollary~\ref{cor:precise:app} and  Remark~\ref{rm:howWeUseIt:alt} below).
Furthermore, to better illuminate the connection
of these ideas to those in \cite{KomorowskiPeszatSzarek2010}, we weaken the condition for convergence at
infinity to one only involving time averages (even if we will not make direct use of
this generalization here). The proof given is essentially the same one
given in \cite{HairerMattinglyScheutzow2011} which in turn closely
mirrors the proofs in \cite{EMattinglySinai2001} and \cite{Mattingly2003}.

In the following section, we also recall a form of the Girsanov theorem which is crucially
used to apply our abstract results.  In particular this formulation is convenient for establishing
the absolute continuity on path space required by the abstract
framework. In the final section, we giving a general recipe for
building an asymptotic coupling leveraging the discussions of the
preceding two sections.

\subsection{Generalities}
\label{sec:gen:erg}

Let $P$ be a family of  Markov transition kernels on
a Polish space $\H$ with a metric  $\rho$. Given a bounded, measurable function
$\phi\colon \H \rightarrow \RR$, we define a new function $P\phi$  on $\H$ by
$P\phi(u) = \int_{\H} \phi(v) P(u,dv)$. For any probability measure
$\nu$ on $\H$, we take $\nu P$ to be the probability measure on $\H$ defined according to
$\nu P (A) =\int_{\H} P(u,A)\nu(du)$ for any measurable $A\subset H$.\footnote{These
left and right actions of $P$ are consistent with the case
when $\H$ is the finite set $\{1,\dots,n\}$ and $P$ is a $n\times n$
matrix given by $P_{ik}=\Prb( \text{ transition }i\rightarrow
j)$. Then $\phi \in \RR^n$ is a column vector and $\nu \in \RR^n$ is a row vector whose
nonnegative entries sum to one. As such, $P \phi$ and $\nu P$ have the
standard meaning given my matrix multiplication.}  Then $\nu P \phi$ is
simply the expected value of $\phi$ evaluated on one step of the
Markov chain generated by $P$ when the initial condition is
distributed as $\nu$.

An invariant measure for $P$ is a probability measure $\mu$
which is a fixed point for $P$ in that $\mu = \mu P$. Since starting the
Markov chain with an initial condition distributed as $\mu$ produces a
stationary sequence of random variables, $\mu$ is also called a
stationary measure for the Markov Chain generated by $P$. An invariant
measure $\mu$ is ergodic if any set $A$ which is invariant for $P$
relative to $\mu$ has $\mu(A) \in \{0,1\}$.\footnote{Recall that a measurable set
  $A$ is invariant for $P$ relative to $\mu$ if $P(u,A)=1$ for
  $\mu$-a.e. $u \in A$ and if $P(u,A)=0$ for $\mu$-a.e.
  $u \not\in A$.  More compactly this say that $A$ is invariant for $P$ relative $\mu$ if $P \indFn{A} = \indFn{A}$, $\mu$-a.e.}
  Notice that the set $\mathcal{I}$ of
invariant measures for $P$ is a convex set.
  It is classical that an invariant measure is
ergodic if and only if it is an extremal point of $\mathcal{I}$. Furthermore, different
ergodic measures are mutually singular. In this sense, the
ergodic invariant measures form the ``atoms'' for the set of invariant
measures as each invariant measure can be written as a convex
combination of ergodic invariant measures and the ergodic invariant
measures cannot be decomposed as a convex combination of other
invariant measures.

  We now lift any probability measure $\mu$ on $\H$ to a canonical
  probability measure on
  the pathspace representing trajectories of the Markov process
  generated by $P$. We denote the pathspace over $\H$ by
\begin{align*}
\H^{\N}=\{ u : \N  \rightarrow \H \}=\{ u =(u_1,u_2,\dots) : u_i \in\H \}
\end{align*}
where $\N=\{1,2,\dots\}$.  The suspension of any initial measure $\nu$ on $H$ to $\H^{\N}$
is denoted by $\nu P^{\N}$.   This pathspace measure $\nu
P^{\N}$ is defined in the standard way from the Kolmogorov
extension theorem by defining the probability  the cylinder sets $\{ U_1 \in A_1,
U_{2} \in A_2,\dots , U_{n} \in A_n\}$ where $n$ is arbitrary.  Here $A_n$ are any measurable subsets of $H$ whereas
$U_0$ is distributed as $\nu$ and given $U_{k-1}$, the $U_{k}$'s are distributed as
$P(U_{k-1},\,\cdot\,)$. The measure $\nu P^{\N}$ should be understood
as the measure on the present state and the entire future trajectory of the Markov chain
$P$  starting from the initial distribution $\mu$. The canonical action of the left shift map
$\theta(u)_j = u_{j +1}$ on $\H^\N$ corresponds to taking one step under
$P$. Hence under $\theta$, the state tomorrow becomes the state
today and  all other
states shift one day closer to the present.

In general a measure $M$ on  $H^\N$
is invariant under the shift $\theta$ if $M\theta^{-1} = M$, where
$M\theta^{-1}$ is defined by $M\theta^{-1}(A) = M(\theta^{-1}(A))$ for
all measurable $A \subset \H^\N$.
Fixing a $\theta$-invariant reference measure
$m$ on $\H^\N$, any other invariant measure $M$ on
$\H^\N$ is ergodic with respect to $(\H^\N,m,\theta)$ if whenever $A \subset
\H^\N$ is invariant for $\theta$ relative to $m$ then $M(A) \in \{0,1\}$.\footnote{Here a measurable set $A$
  is invariant for $\theta$ relative to a probability measure $m$ on $H^\N$
  if $\theta^{-1}(A) = A \text{ mod }m$,
  which is to say the  symmetric difference   $\theta^{-1}(A) \Delta
  A$ is measure zero for $m$.} As in the previous setting, an
invariant measure is ergodic if and only if it is an extremal point of
the set of invariant measures for $\theta$ on $\H^\N$.

The notions of ergodic and extremal measures
on $\H$ and $\H^\N$ are
self consistent in that the following statements are equivalent:
\begin{itemize}
\item[(i)] a measure $\mu$ on $\H$ is an ergodic invariant measure for $P$
\item[(ii)] $\mu P^{\N}$ is an  ergodic invariant measure on $H^{\N}$ relative to the shift map $\theta$
\item[(iii)] $\mu P^{\N}$ is an extremal point for the set of $\theta$-invariant measures  on $\H^\N$
\item[(iv)] $\mu$  is an extremal point for the set of $P$-invariant measures on $\H$.
\end{itemize}
For further discussion of all  these ergodic theory generalities see \cite{sinai87,Kallenberg02,ZabczykDaPrato1996}

\subsection{Equivalent Asymptotic Couplings and Unique Ergodicity}
\label{sec:EAC}
 We say that a probability measure $\Gamma$ on  $\H^{\N} \times
 \H^{\N}$ is an \emph{asymptotically equivalent coupling} of
 two measures $M_1$ and $M_2$ on $\H^\N$ if $\Gamma
 \Pi_i^{-1} \ll M_i$, for $i = 1,2$, where $\Pi_1(u,v) = u$ and $\Pi_2(u,v) = v$ . We
 will write $\tilde{\mathcal{C}}(M_1,M_2)$ for the set of all such
 asymptotically equivalent couplings.

Given any bounded (measurable) function $\phi\colon \H\rightarrow \RR$, we define
$\bar D_\phi \subset \H^{\N} \times
\H^{\N}$ by
\begin{align}
  \bar D_\phi := \Big\{ (u,v) \in \H^{\N} \times
\H^{\N}: \lim_{n\rightarrow \infty}\frac1n \sum_{k=1}^n  ( \phi(u_k) -
  \phi(v_k) ) =0 \Big\}\,.
  \label{eq:b:d:ave}
\end{align}
On the other hand a set $\mathcal{G}$ of bounded, real-valued, measurable
functions on
$H$ is said to \emph{determine measures} if, whenever $\mu_1,
\mu_2 \in Pr(H)$ are such that $\int \phi d \mu_1 = \int \phi d
\mu_2$ for all $\phi \in \mathcal{G}$, then $\mu_1 = \mu_2$.
\begin{theorem}\label{thm:asCoupling} Let $\mathcal{G}\colon H
  \rightarrow \RR$ be collection
  of functions which determines measures.
  Assume that there exists a measurable $\H_0 \subset \H$ such
  that for any $(u_0, v_0) \in \H_0\times \H_0$ and any $\phi
  \in \mathcal{G}$ there exists a $\Gamma = \Gamma(u_0, v_0, \phi)
  \in \tilde{\mathcal{C}}(\delta_{u_0}P^{\N},
  \delta_{v_0}P^\N)$ such that $\Gamma(\bar D_\phi) >0$.
  Then there exist at most one ergodic invariant measure $\mu$ for
  $P$ with $\mu(\H_0)>0$. In particular if $H_0=H$, then there exists at most one,
  and hence ergodic,
  invariant measure.
\end{theorem}

 \begin{remark}
   At first glance it might be surprising that equivalence is sufficient to
   determine the long time statistics. However since the Birkhoff
   Ergodic theorem implies that time averages
   along typical trajectories converge to the integral against an
   ergodic invariant measure, one only needs to draw typical infinite
   trajectories. From this it is clear that  absolutely continuity
   on the entire future trajectory indexed by $\N$  is critical and can not be replaced
   with absolutely continuity on $\{1,2,\dots,N\}$ for all $N \in \N$.
   Since absolutely continuous measures have the same paths, only
   with different weights, we see that absolutely continuous measures
   are sufficient to ensure one is drawing typical trajectories.
 \end{remark}

\begin{proof}[Proof of Theorem~\ref{thm:asCoupling}] Assume that there
  are two ergodic invariant probability measures $\mu$ and $\nu$ on
  $\H$ such that both $\mu(\H_0)>0$ and $\nu(\H_0)>0$. Fix any $\phi
  \in \mathcal{G}$. By Birkhoff's ergodic theorem there exists sets
  $A_\phi^\mu,A_\phi^\nu \subset \H^{\N}$ such
  that $\mu P^\N (A_\phi^\mu) =\nu P^\N (A_\phi^\nu) =1$ and such that if $u \in
  A_\phi^\mu$ and  $v \in
  A_\phi^\nu$ then
  \begin{align}
    \lim_{n \rightarrow \infty} \frac1n \sum_{k=1}^n \phi(u_k)  =
    \int_{\H} \phi d\mu \quad\text{and}\quad  \lim_{n \rightarrow \infty} \frac1n \sum_{k=1}^n \phi(v_k)  =
    \int_{\H} \phi d\nu \,.
    \label{eq:b:e:t:conv}
  \end{align}
  Define
  $A_\phi^\mu(u_0) = \{ \tilde u = (\tilde{u}_0, \tilde{u}_1, \ldots) \in A_\phi^\mu : \tilde u_0 = u_0\}$
  and $A_\phi^\nu(u_0)$ analogously.   Notice that $\delta_{u_0} P^\N( A_\phi^\mu(u_0)) = \delta_{u_0}P^\N(A_\phi^\mu)$
  for any $u_0 \in H$.
  Hence, by Fubini's theorem we have, for $\mu$-a.e. $u_0 \in H$,
  that $\delta_{u_0} P^\N( A_\phi^\mu(u_0)) = \delta_{u_0}P^\N(A_\phi^\mu) = 1$, and that
  $\delta_{v_0}P^\N(A_\phi^\nu(v_0)) = 1$   for $\nu$-a.e. $v_0 \in H$.

  Since we have presumed that $\H_0$ is non-trivial relative to both $\mu, \nu$, we may now
  select a pair of initial conditions $u_0, v_0 \in \H_0$
  such that $\delta_{u_0}P^{\N}(A_\phi^\mu(u_0)) =\delta_{v_0}P^{\N}(A_\phi^\nu(v_0))=
  1$.  Let $\Gamma$ be the measure in $\tilde{\mathcal{C}}(\delta_{u_0}P^{\N},
  \delta_{v_0}P^{\N})$ given in the assumptions of the
  Theorem corresponding to these initial points $u_0$, $v_0$
  and the test function $\phi$.
  Since
  $\Gamma \Pi_1^{-1} \ll
  \delta_{u_0}P^{\N}$ and $\Gamma \Pi_2^{-1} \ll
  \delta_{v_0}P^{\N}$, where again $\Pi_1(u,v)=u$ and  $\Pi_2(u,v)=v$, we have that
  $\Gamma(A_\phi^\mu(u_0) \times
  A_\phi^\nu(v_0))=1$. Defining $\bar D_\phi' = \bar D_\phi
  \cap \big( A_\phi^\mu(u_0) \times
  A_\phi^\nu(v_0) \big)$, one has $\Gamma(\bar D_\phi') >0$
  and thus we infer that $\bar D_\phi'$ is nonempty. Observe that for any $(u,v) \in \bar D_\phi'$,
  in view of \eqref{eq:b:d:ave}, \eqref{eq:b:e:t:conv} and the definition of $\bar D_\phi'$, we have
  \begin{align*}
      \int_{\H} \phi \, d\mu - \int_{\H} \phi \, d\nu =   \lim_{n \rightarrow \infty} \frac1n
    \sum_{k=1}^n (\phi(u_k) -\phi(v_n))   = 0   \,.
  \end{align*}
  Since $\phi$ was an arbitrary function in $\mathcal{G}$, which was
  assumed to be
  sufficiently rich to determine measures, we conclude that $\mu_1 =
  \mu_2$ and the proof is complete.
\end{proof}

We next provide a simple corollary of Theorem~\ref{thm:asCoupling}
to be used directly in the examples provided below.  To this
end, we consider a possibly different distance $\tilde{\rho}$ on $H$ and define
\begin{align*}
  D_{\tilde{\rho}} := \left\{ (u,v) \in H^{\N} \times H^{\N}: \lim_{n \to \infty} \tilde{\rho}(u_n, v_n) = 0\right\}.
\end{align*}
We also consider the class of test functions
\begin{align*}
  \mathcal{G}_{\tilde{\rho}} = \Big\{  \phi \in C_b(H): \sup_{u \not = v} \frac{ |\phi(u) - \phi(v)|}{\tilde{\rho}(u,v)} < \infty\Big\}.
\end{align*}
The corollary is as follows:
\begin{corollary}\label{cor:precise:app}
Suppose that $\mathcal{G}_{\tilde{\rho}}$ determines measures on $(H, \rho)$ and assume that $D_{\tilde{\rho}}$ is a
measurable subset of $H^{\N} \times H^{\N}$.    If $H_0 \subset H$ is a measurable set such that for each pair $u_0, v_0 \in H_0$
there exists an element $\Gamma \in \tilde{\mathcal{C}}(\delta_{u_0}P^{\N}, \delta_{v_0}P^\N)$ with $\Gamma(D_{\tilde{\rho}})> 0$,
then there exists at most one ergodic invariant measure $\mu$ with $\mu(H_0) > 0$.
\end{corollary}
\begin{proof}
With the observation that
\begin{align*}
  D_{\tilde{\rho}} \subset \bar{D}_{\phi}  \textrm{ for every }  \phi \in \mathcal{G}_{\tilde{\rho}},
\end{align*}
the desired result follows immediately from Theorem~\ref{thm:asCoupling}.
\end{proof}
\begin{remark}\label{rm:howWeUseIt:alt}
   The conditions imposed on $\tilde{\rho}$, $\mathcal{G}_{\tilde{\rho}}$ and $D_{\tilde{\rho}}$
   in Corollary~\ref{cor:precise:app} are easily verified in practice.  For instance, consider
   $H = H^1(\TT^d)$ with the usual topology and take $\tilde{\rho}(u_0, v_0)  = \|u_0 - v_0 \|_{L^2}$,
   the $L^2(\TT^d)$ distance.  Since $\tilde{\rho}$ is continuous on $H^1$ we see that $D_\rho$ is measurable.
   Furthermore, a simple mollification argument allows one to show that
   $\mathcal{G}_{\tilde{\rho}}$ is determining.  Similar considerations will
   allow us to directly apply Corollary~\ref{cor:precise:app} in each of the examples below.
\end{remark}

\subsection{Girsanov's Theorem Though a Particular Lens}

Let $\{W_k(t): k=1,\dots,d\}$ be a collection of independent one-dimensional Brownian motions and define
$W(t)=(W_1(t),\dots,W_d(t))$. Take $h(t)$ to be an $\RR^d$-valued stochastic
process adapted to the filtration generated by $W(t)$ such
that
\begin{align}
  \int_0^\infty |h(s)|^2 ds  \leq C \quad\text{almost surely,}
  \label{eq:st:nov:cond}
\end{align}
for some finite (deterministic) constant $C$. Now define $\widetilde{W}$ by
\begin{align*}
   \widetilde{W}(t) = W(t) + \int_0^t h(s) ds
\end{align*}
for all $t\geq 0$. The following result is a restating of the Girsanov
theorem (see e.g. \cite{RevuzYor1999} for further details):
\begin{theorem}\label{thm:GirsonovRedux} In the above setting, the law of $\widetilde{W}$ is
  equivalent to that of $W$ as measures on $C([0,\infty),
  \RR^d)$. Furthermore if $\Phi$ is a measurable map from  $C([0,\infty),
  \RR^d)$ into $H^\N$ for some Polish space $\H$, then the law of
  $\Phi(\widetilde{W})$ is equivalent to that of  $\Phi(W)$ as
  measures on $\H^\N$.
\end{theorem}

\begin{remark}
  The assumption given in \eqref{eq:st:nov:cond} is an overkill, as
  Girsanov's Theorem holds under much less restrictive
  assumptions. However, in all of our application
  \eqref{eq:st:nov:cond} will hold.
\end{remark}

\subsection{A Recipe for  Asymptotic Coupling}
\label{sec:recipeC}
We now outline the basic logic used to in all of our examples. To
avoid technicalities arising in the infinite dimensional setting such
as the domain of various operators, we begin with an example in finite
dimensions. However, as we will see, these arguments directly apply
to infinite dimensional SPDEs when all of the objects involved are
well defined.

Consider the stochastic differential equation given by
\begin{align}\label{sde}
  dx = F(x)dt + \sigma dW \quad\text{with}\quad x(0)=x_0 \in \RR^d\,,
\end{align}
where $F\colon \RR^d \rightarrow \RR^d$, $W$ is an $n$-dimensional
Brownian Motion, and $\sigma$ is a $d\times n$-dimensional matrix
chosen so that \eqref{sde} has global solutions.  Fundamentally, the question of unique ergodicity
turns on showing that \eqref{sde} and a second copy
\begin{align}\label{sde-y}
  dy = F(y)dt + \sigma d\widetilde{W} \quad\text{with}\quad y(0)=y_0 \in \RR^d\,,
\end{align}
have identical long time statistics even when $x_0 \neq y_0$. Since we
are only interested in showing that the marginals of $(x,y)$ have the
same statistics, we are free to couple the two Brownian motions
$(W,\widetilde{W})$ in any way we wish, building in correlations which
are useful in the analysis. The essence of
Corollary~\ref{cor:precise:app} is that we can even replace
$\widetilde{W}$ with another process as long as it is absolutely
continuous with respect to a Brownian Motion on the infinite time
horizon, namely $C([0,\infty); \RR^n)$. To this end, we use $\widetilde{y}$
to represent a solution to \eqref{sde-y} driven by this modified
process $\widetilde{W}$, and take
$\widetilde{W}$ to be a Brownian Motion shifted with a feedback
control $G$ which depends on the current state of $(x,\widetilde{y})$ and is
designed to send $|x(t)-\widetilde{y}(t)| \rightarrow 0$ as $t\rightarrow \infty$. One
could consider more general adapted controls (or even non-adapted as
in \cite{HairerMattingly06,HairerMattingly2011}\footnote{Of course,
  non-adapted controls make things significantly more technical. In
  particular, the classical Girsanov Theorem can not be used.} ) but this class has
proven sufficient for all of the problems we present here. In addition
to the control $G$, we will introduce a stopping time $\tau$ which will turn
off the control should $(x(t),\widetilde{y}(t))$ separate too much.
This stopping time ensures that \eqref{eq:st:nov:cond} holds and hence
guarantees that our shifted process $\widetilde{y}$ is absolutely continuous
with respect to $y$.

In light of this discussion,  consider the system
\begin{align*}
  dx &= F(x)dt + \sigma dW\quad\text{with}\quad {x(0)=x}_0 \in \RR^d\,,\\
d\widetilde{y} &= F(\widetilde{y})dt  + G(x,\widetilde{y}) \one{t \leq
                 \tau} dt+ \sigma dW\quad\text{with}\quad y(0)=\widetilde{y}_0 \in \RR^d\,,
\end{align*}
where $G\colon \RR^d\times \RR^d \rightarrow \RR^d$ is our feedback control and $\tau$ is a
stopping time adapted to the filtration generated by $\{ (x_s,
\widetilde{y}_s) : s \leq t\}$. We assume that $G$ and $\tau$ are such
that the system $(x,\widetilde{y})$ has global solutions. Furthermore,
we assume that  everything is constructed so that $\Prb( \tau = \infty) >0$ and $|x(t) - \widetilde{y}(t)|
\rightarrow 0$ on the event $\{\tau= \infty\}$. If in addition,
\begin{align}\label{toyBound}
  \int_0^\infty   |\sigma^{-1}G(x(t),\widetilde{y}(t))|^2 \one{t \leq \tau}\, dt < C\quad\text{a.s.,}
\end{align}
for some deterministic (finite) $C>0$, then by Theorem~\ref{thm:GirsonovRedux}
\begin{align*}
   \widetilde{W}(t) = W(t) + \int_0^t \sigma^{-1}G(x(s),\widetilde{y}(s))\one{s \leq \tau} ds
\end{align*}
is equivalent to a Brownian motion on $C([0,\infty), \RR^n)$. Implicit
in equation \eqref{toyBound} is the assumption that the range of $G$
is contained in the range of $\sigma$.\footnote{$\sigma$ need not be
  invertible. As long as the range of $G$ is contained in the range of
  $\sigma$ then $\sigma^{-1}$ can be taken to be the pseudo-inverse.}  Moreover, by Theorem~\ref{thm:GirsonovRedux} and \eqref{toyBound}, we also
  see that $\widetilde{y}$ is
equivalent to the solution of \eqref{sde-y} on $C([0,\infty), \RR^d)$. In summary, if
$P\phi(x_0) = \E \phi(x(t,x_0))$ for some $t>0$ where $x(t,x_0)$ solves
\eqref{sde} (starting from the initial condition $x_0 \in \RR^d$), $H_0=H=\mathbb{R}^d$,
and $\delta_{x_0}P^\N$ is defined as in Section~\ref{sec:gen:erg}, then
the law induced by $\{ (x(nt,x_0), \widetilde{y}(nt, y_0)) : n=1,2\dots\}$ is an
element of
$\widetilde{\mathcal{C}}( \delta_{x_0}P^\N,
\delta_{\widetilde{y}_0}P^\N)$ which charges $\mathcal{D}_{|\cdot|}$ with positive probability.
And hence by Corollary~\ref{cor:precise:app}, we know that \eqref{sde}
has at most one invariant measure.

The question remains, which $G$ and $\tau$ do we choose? There is no unique choice.
One only needs to ensure that the range of $G$ is contained in the
range of $\sigma$, that \eqref{toyBound} holds, and that $|x(t) - \widetilde{y}(t)|
\rightarrow 0$ on the event ${\tau=\infty}$. Typically, we will take $\tau =
\inf\{ t >0: \int_0^t |  \sigma^{-1}G(x_s,\widetilde{y}_s)|^2 ds \geq R\}$
  for some $R>0$.  To explore this question informally, we define
  $\rho(t)=x(t)-\widetilde{y}(t)$ and observe that
  \begin{align*}
    \frac{d}{dt}\rho(t) = F(x) -F(\widetilde{y}) - G(x,\widetilde{y})\one{t \leq \tau} ,
  \end{align*}
provided $\sigma^{-1}G$ is well defined on the interval $[0,\tau]$. If $\sigma$ is invertible, one choice is to take $G(x,\widetilde{y})=
F(x) -F(\widetilde{y}) + \lambda (x-\widetilde{y})$ which results in
\begin{align*}
    \frac{d}{dt}\rho(t) = - \lambda\rho
\end{align*}
for $t<\tau$, so that $|x(t) - \widetilde{y}(t)|$ clearly decays towards zero on $[0,\tau)$ with a rate independent of $R$. Furthermore, \eqref{toyBound} holds
provided $F$  does not grow too fast and has some H\"older
regularity, and $\tau=\infty$ almost surely provided $R>0$ is chosen sufficiently large.
Often taking $G(x,\widetilde{y})=\lambda (x-\widetilde{y})$ is sufficient for a
$\lambda$ large enough (see
\cite{Hairer02,HairerMattinglyScheutzow2011} for example).
 If only part of $F$ leads to instability, say $\Pi F$ for
some projection $\Pi$, then we can often take $G(x,\widetilde{y})=
\Pi( F(x) -F(\widetilde{y}) )+ \lambda \Pi(x-\widetilde{y})$ or even simply
$G(x,\widetilde{y})=\lambda \Pi(x-\widetilde{y})$, and only assume that the range of $\sigma$
contains the range of $\Pi$. This loosening of the assumptions on the range of $\sigma$
is one of the principle advantages of this point of view for SPDEs.\footnote{
Taking $G(x,\widetilde{y})=
F(x) -F(\widetilde{y}) + \lambda \tfrac{x-\widetilde{y}}{|x-\widetilde{y}|}$ leads to $\rho$
dynamics which converge to zero in finite time. This can be used to
prove convergence in total variation norm. However it is less useful
when one takes  $G(x,\widetilde{y})=
\Pi(F(x) -F(\widetilde{y}) + \lambda \tfrac{x-\widetilde{y}}{|x-\widetilde{y}|})$ as the remaining
degrees of freedom only contract asymptotically at $t \rightarrow
\infty$. Nonetheless, such a control can simplify the convergence analysis in some cases.}

\section{Examples}
\label{sec:examples}

\subsection{Navier-Stokes on a Domain}
\label{sec:NSE:Channel}

Our first example is the 2D stochastic Navier-Stokes equation
\begin{align} \label{eq:nse}
  &d\bfU + \bfU \cdot \nabla \bfU dt =  (\nu \Delta \bfU - \nabla
    \pres + \mathbf{f})dt  + \sum_{k =1}^d \sigma_k dW^k, \quad \nabla \cdot\bfU = 0,
\end{align}
for an unknown velocity field $\bfU = (u_1, u_2)$ and pressure $\pres$ evolving on a bounded domain $\DD \subset \RR^2$ where we assume that
$\partial \DD$ is smooth and $\bfU$ satisfies the no-slip (Dirichlet) boundary condition
\begin{align}
  \bfU_{| \partial \DD} = 0.
    \label{eq:nse:bc}
\end{align}
Here, in addition to the given vector fields $\sigma_j \in L^2(\DD)$
and a corresponding collection of $W = (W_1, \ldots, W_d)$ independent
standard Brownian motions,
the dynamics of \eqref{eq:nse}--\eqref{eq:nse:bc}
are also driven by a fixed, deterministic $\mathbf{f} \in L^2(\DD)$.
We refer to e.g. \cite{ConstantinFoias1988, Temam2001} and to \cite{CIME08}
for further details on the mathematical setting of the Navier-Stokes
equations in the deterministic and stochastic frameworks respectively.

\begin{remark}
  If either $\mathbf{f}$ or any of the $\sigma_k$ are not divergence free, they
  can be replaced with their projection onto the divergence free
  vector fields without changing the dynamics as this only changes
  the pressure which acts as a Lagrange multiplier in this setting,
  keeping solutions on the space of divergence free vector fields.
\end{remark}

\subsubsection{Mathematical Preliminaries}
We consider \eqref{eq:nse} on the phase space
\begin{align*}
  H :=  \{ \bfU \in L^2(\DD)^2: \nabla \cdot \bfU = 0, \bfU \cdot \mathbf{n} = 0 \},
\end{align*}
where $\mathbf{n}$ is the outward normal to $\partial \DD$.
Denote $P_L$ as the orthogonal projection of $L^2(\DD)^2$ onto $H$.
The space of vector fields whose gradients are integrable in $L^2(\DD)^2$ are also relevant and we define
$V :=  \{ \bfU \in H^1(\DD)^2: \nabla \cdot \bfU = 0, \bfU_{| \partial \DD} = 0 \}$.
We denote the norms associated to $H$ and $V$ respectively as $| \cdot |$ and $\| \cdot \|$.

The Stokes operator is defined as $A\bfU = - P_L \Delta \bfU$, for any vector field $\bfU \in V \cap H^2(\DD)^2$.    Since $A$ is self-adjoint
with a compact inverse we infer that $A$ admits an increasing sequence of eigenvalues $\lambda_k \sim k$ diverging to infinity with the corresponding eigenvectors
$e_k$ forming a complete orthonormal basis for $H$.  We denote by $P_N$ and $Q_N$ the projection onto $H_N = \mbox{span}\{ e_k: k =1, \ldots N\}$
and its orthogonal complement, respectively.  Recall the generalized Poincar\'e inequalities
\begin{align}
   \| P_N \bfU\|^2 \leq \lambda_N |P_N \bfU|^2  \quad  | Q_N \bfU|^2 \leq \lambda_N^{-1}  \| Q_N \bfU\|^2
   \label{eq:gen:poincare}
\end{align}
hold for all sufficiently smooth $\bfU$ and any $N \geq 1$.

Recall that for all $\bfU_0 \in H$ and any fixed, finite $d$,  \eqref{eq:nse} admits a unique
solution
\begin{align*}
\bfU(\spcdot) = \bfU(\spcdot, \bfU_0) \in L^2(\Omega; C([0,\infty); H) \cap L^2_{loc}([0,\infty); V)),
\end{align*}
which depends continuously in $H$ on $\bfU_0$ for each $t \geq 0$.  As such the transition functions $P_t(A, \bfU_0) = \Prb(\bfU(t,\bfU_0) \in A)$ are well defined
for any $\bfU_0 \in H$, $t \geq 0$ and any Borel subset $A$ of $H$, and define an associated Feller Markov semigroup $\{P_t\}_{t\geq 0}$
on $C_b(H)$.

We next recall some basic energy estimates for \eqref{eq:nse}. Applying the It\={o} lemma to \eqref{eq:nse} we find that
\begin{align*}
  d | \bfU |^2 + 2\nu \| \bfU\|^2 dt = 2\langle \mathbf{f}, \bfU \rangle dt +  |\sigma|^2 dt +  2 \langle \sigma, \bfU\rangle dW
\end{align*}
where, for any $\mathbf{v} \in H$,  $\langle \sigma, \mathbf{v} \rangle : \RR^d \to \RR$ is the linear operator defined by $\langle \sigma, \mathbf{v} \rangle w := \sum_{k = 1}^n
\langle \sigma_k, \mathbf{v} \rangle w_k$, $w \in \RR^d$ and $|\sigma|^2 := \sum_{k =1}^n |\sigma_k|_{L^2}^2$ is the mean
instantaneous energy injected into the system per unit time.
Thus, for $R \geq 0$, using exponential martingale estimates\footnote{Recall that for any continuous martingale
$\{M(t)\}_{t\geq 0}$,
\begin{align}
\Prb\big( \sup_{t \geq 0} \  M(t) - \gamma \langle M \rangle(t) \, \geq\, R \big) \leq e^{-\gamma R}
\label{eq:exp:mart}
\end{align}
for any $R, \gamma >0$ where $ \langle M \rangle(t)$ is the quadratic variation of $M(t)$.}
we infer that, for $\alpha = \alpha(|\sigma|, \nu) = \frac{\nu}{|\sigma |^2}$, independent of $R$
\begin{align}
  \Prb \Big( \sup_{t \geq 0} \  |\bfU(t)|^2 +
  \nu \int_0^t \|\bfU\|^2 ds -
  (| \sigma |^2 + \tfrac{|A^{-\frac{1}{2}}f|^2}{2\nu} ) t - |\bfU_0|^2  \,\geq\, R \Big) \leq \exp(-\alpha R).
  \label{eq:exp:mg:est}
\end{align}
Note that \eqref{eq:exp:mg:est} implies that time averaged measures
$\nu_T(A) =\frac{1}{T} \int_0^T  \Prb(\bfU(s) \in A) ds$ are a tight sequence
since $V$ is compactly embedding into $H$. Since $\bfU_0 \mapsto \E
\phi(\bfU(t, \bfU_0))$ is continuous and bounded in $L^2$ whenever $\phi$ is (namely
the Markov semigroup is Feller), the Krylov-Bogolyubov theorem implies the collection of invariant measures
corresponding to \eqref{eq:nse} is non-empty.

\subsubsection{Asymptotic Coupling Arguments}
Having now reviewed the basic mathematical setting of \eqref{eq:nse},
the uniqueness of invariant measures corresponding to \eqref{eq:nse}
is established using the asymptotic coupling framework introduced above.  Fix any
$\bfU_0, \tbfU_0 \in H$ and consider
$\bfU(\spcdot) = \bfU(\spcdot, \bfU_0)$ solving \eqref{eq:nse} with
initial data $\bfU_0$, and $\tbfU$ solving
\begin{align}
  &d\tbfU + \tbfU \cdot \nabla \tbfU dt =  (\nu \Delta \tbfU +
    \indFn{\{\tau_K > t\}} \lambda P_N( \bfU - \tbfU) +  \nabla
    \tilde{\pres}+ f)dt  + \sum_{k =1}^d \sigma_k dW^k,
  \quad \nabla \cdot\tbfU = 0, \ \ \tbfU(0) = \tbfU_0,
  \label{eq:nse:shift}
\end{align}
where
\begin{align*}
  \tau_K := \inf_{t \geq 0}  \left\{   \int_0^t |P_N  (\bfU - \tbfU)|^2 ds \geq K \right\}
\end{align*}
and $K, \lambda> 0$ are fixed positive parameters which we will
specify below as a function of $\bfU_0, \tbfU_0$. In the context of
the framework presented in Section~\ref{sec:recipeC},
$G(\bfU ,\tbfU ) = \lambda P_N( \bfU - \tbfU)$ and the stopping time
is $\tau_K$.

We now make the connection with Corollary~\ref{cor:precise:app}
explicit.  Fix $T>0$ and take $t_n=nT$. Define the measures  $\mathfrak{m}$ and
$\mathfrak{n}$ on $H^\N$ to be, respectively, the laws of the random
vectors
\begin{align*}
  \big( \bfU(t_1, \tilde{\bfU}_0 ) ,\bfU(t_2, \tilde{\bfU}_0 )
  ,\dots\big)  \qquad \text{and}\qquad \big( \tbfU(t_1, \tilde{\bfU}_0 ) ,\tbfU(t_2, \tilde{\bfU}_0 )
  ,\dots\big) \,.
\end{align*}
The Girsanov theorem as presented in
Theorem~\ref{thm:GirsonovRedux}, implies that $\mathfrak{n}$ is mutually absolutely continuous with
respect to $\mathfrak{m}$.  Indeed, let $h(t) =\indFn{\{\tau_K > t\}} \lambda \sigma^{-1}  P_N( \bfU -\tbfU)$,
where $\sigma^{-1}$ is the psuedo-inverse of $\sigma$.
Thanks to the  definition of the stopping times $\tau_K$, we have, for any
choice of $\lambda>0$ and $K >0$, that $h$ satisfies the condition \eqref{eq:st:nov:cond}. We again
emphasize
that the equivalence of the measures $\mathfrak{m}$ and $\mathfrak{n}$
holds on the entire infinite trajectory sampled at the times
$\{T,2T, \dots\}$ which is significantly stronger than absolute continuity for the
trajectories sampled at finite number of times $\{T,2T, \dots,nT\}$ for
all $n>0$.

We now define the measure $\Gamma$ on the space $H^\N \times H^\N$ as
the law of the random vector
\begin{align*}
   \big(\bfU(t_n,\bfU_0) , \tbfU(t_n,\tbfU_0)\big)_{n \in \N}\,.
\end{align*}
In view of the discussions in the previous paragraph, for any $\lambda, K > 0$, $\Gamma$ is an
element $\tilde{C}(\delta_{u_0}P^{\N}, \delta_{v_0}P^\N)$.
The uniqueness of invariant measures corresponding to
\eqref{eq:nse} therefore follows immediately from Corollary~\ref{cor:precise:app} if, for each
$\bfU_0, \tbfU_0 \in H$, we can find a corresponding $\lambda, K>0$
(which may well depend on $\bfU_0, \tbfU_0 \in H$) such that
$\bfU(t) - \tbfU(t) \to 0$ in $H$ as $t \to \infty$ on a set of
nontrivial probability.

Take $\bfV = \bfU- \tbfU$ and $\dpres =  \pres - \tilde{\pres}$. We have that
\begin{align}
  \pd{t} \bfV - \nu \Delta \bfV + \indFn{\{\tau_K > t\}} \lambda  P_{N} \bfV = - \nabla q -  \bfV\cdot \nabla \bfU - \tbfU \cdot \nabla \bfV, \quad \nabla \cdot \bfV = 0\,.
  \label{eq:diff:eq:1}
\end{align}
Taking $\lambda = \nu \lambda_{N}$, the Poincar\'e inequality,
\eqref{eq:gen:poincare} implies that
\begin{align}
  \indFn{\{\tau_K > t\}}
  \lambda_N \nu |P_{N} \bfV|^2+ \nu \| \bfV\|^{2}  \geq \indFn{\{\tau_K > t\}}
  (\nu \lambda_N   |P_{N} \bfV|^2+ \nu\| Q_{N}\bfV\|^{2} ) = \indFn{\{\tau_K > t\}}
  \nu   \lambda_N|\bfV|^2\,.
  \label{eq:inv:pon:app}
\end{align}
Multiplying \eqref{eq:diff:eq:1} with $\bfV$, integrating over $\DD$,
using that $\bfV, \bfU, \tbfU$ are all divergence free and \eqref{eq:inv:pon:app} we obtain
\begin{align*}
  \frac{d}{dt} |\bfV|^{2} + \nu \| \bfV\|^{2} + \lambda_N \nu \indFn{\{\tau_K > t\}} |\bfV|^{2}
  \leq  \left| \int_{\DD} \bfV \cdot \nabla \bfU \cdot \bfV dx \right| \leq C | \bfV | \| \bfV \| \| \bfU \|
  \leq \nu \| \bfV \|^{2} + C_1| \bfV |^{2} \| \bfU \|^{2}
\end{align*}
where $C_1$ depends only on $\nu$ and universal quantities from Sobolev embedding.
Rearranging and using the Gr\"onwall lemma we obtain
\begin{align}
   |\bfV(t)|^{2} \leq |\bfU_{0} - \tbfU_{0}|^{2}\exp\Big(  -\lambda_N \nu t + C_1  \int_{0}^{t}\| \bfU \|^{2}ds\Big),
   \label{eq:grone:wall:conclusion}
\end{align}
for any $t \in [0, \tau_K]$.

Now, for any $R > 0$, consider the sets
\begin{align*}
   E_{R} := \left\{  \sup_{t \geq 0} \Big( |\bfU(t)|^2 +  \nu  \int_0^t \| \bfU\|^2 ds - (| \sigma |^2 +\tfrac{|A^{-\frac12}f|^2}{2 \nu}) t - |\bfU_0|^2 \Big) < R\right\}.
\end{align*}
Notice that, in view of \eqref{eq:exp:mg:est}, these sets have nonzero probability
for every $R =R(\nu,|\sigma|)> 0$ sufficiently large.
On the other hand, on $E_{R}$, \eqref{eq:grone:wall:conclusion} implies
\begin{align*}
   |\bfV(t)|^{2}
   \leq |\bfU_{0} - \tilde{\bfU}_{0}|^{2} \exp\Big(\frac{C_1}{\nu} (R + |\bfU_{0}|^{2}) \Big)
   \exp\Big(  -\lambda_N \nu t + \frac{C_1}{\nu}
  (| \sigma|^2 + \tfrac{|A^{-\frac12}f|^2}{4\nu} )t  \Big).
\end{align*}
for each $t \in [0, \tau_K]$.  Note carefully that the constant $C_1$
is independent of $N, \lambda, K >0$
and $\bfU_{0}, \tilde{\bfU}_{0}$.  By picking $N$ such that
\begin{align}
  \frac{\nu \lambda_{N}}{2} - \frac{C_1}{\nu}\big( | \sigma|^2 +  \frac{|A^{-\frac12}f|^2}{2\nu} \big) > 0,
  \label{eq:the:lamb:choice}
\end{align}
we infer, for $\lambda = \nu \lambda_N$,
\begin{align}
   |\bfV(t)|^{2}
   \leq |\bfU_{0} - \tilde{\bfU}_{0}|^{2} \exp\Big(\frac{C_1}{\nu} (R + |\bfU_{0}|^{2}) \Big)
   \exp\Big(  -\frac{\lambda}{2} t\Big),
   \label{eq:good:decay:final}
\end{align}
on $E_{R}$ for every $t \in [0, \tau_K]$ and where we again emphasize that $C_1$ does not depend on $K$ in \eqref{eq:nse:shift}.
By now choosing $K = K(\nu,  |\sigma|^2)$ sufficiently large we are forced to conclude from \eqref{eq:good:decay:final} that $\{\tau_K = \infty\} \supset E_R$
and hence on the non-trivial set $E_R$ we infer that
\begin{align*}
  \bfU(t) - \tbfU(t) \longrightarrow 0 \quad\textrm{ in }\quad H.
\end{align*}
as $t \to \infty$.

In conclusion we have proven that
\begin{proposition}
  \label{prop:nse:unique}
  For every $\nu > 0$ there exists $N = N(\nu, | \sigma |^2,|A^{-\frac12}f|)$ such
  that if $\mathrm{Range}(\sigma) \supset H_N=P_N H$ then
  \eqref{eq:nse} has a unique ergodic invariant measure.
\end{proposition}

\begin{remark}
  It is worth emphasizing here that the above analysis shows that unique ergodicity results can be easily obtained
  in the presence of a deterministic forcing.    This observation applies to each
  of the examples considered below, but we omit its explicit inclusion for brevity
  and clarity of presentation.  Of course, the addition of a body
  forcing $\mathbf{f}$ in the hypo-elliptic setting can bring extra
  complications, primarily in proving topological irreducibility which
  is often required to prove unique ergodicity.
\end{remark}

\subsection{2D Hydrostatic Navier-Stokes Equations}\label{sec:SPE}
We next consider a stochastic version of the 2D Hydrostatic Naver-Stokes equations
\begin{align}
   &d u + (u \partial_x u + w \partial_z u
                 + \partial_x p
                 - \nu \Delta u)dt  =
                  \sum_{k=1}^d \sigma_k dW^k
      \label{eq:PEFromalEQ1Momentum}\\
    & \partial_z p = 0
      \label{eq:Hydrostatic:pres}\\
    &\partial_x u + \partial_z w  = 0,
      \label{eq:PEFromalEQ1DivFree}
\end{align}
for an unknown velocity field $(u, w)$ and pressure $p$ evolving
on the domain $\DD = (0,L) \times (-h, 0)$.
The boundary $\partial \DD$ is decomposed into its vertical sides $\Gamma_v = [0,L] \times \{0,-h\}$ and lateral sides $\Gamma_l = \{0 ,L\} \times [-h, 0]$,
 and we impose the boundary conditions
\begin{align}
  u = 0  \textrm{ on } \Gamma_l,
  \quad \partial_ zu = w = 0  \textrm{ on } \Gamma_v. \label{eq:HNSE:bcs}
\end{align}
The system is driven by a collection of independent Brownian motions $(W^1,\ldots,W^d)$ acting in directions $\sigma_k\in L^2(\DD)$ to be specified below.

The hydrostatic Navier-Stokes equations serve as a simple mathematical model which
maintains some of the crucial anisotropic structure present in the more involved Primitive
equations of the oceans and atmosphere.  This latter system forms the numerical core of sophisticated general circulation
models used in climate and weather prediction \cite{pedlosky2013,Trenberth1992}.  The Primitive equations
have been studied extensively in the mathematics literature in both deterministic \cite{LionsTemamWang1,LionsTemamWang2,LionsTemamWang3,CaoTiti2007,
Kobelkov2007, ZianeKukavica, PetcuTemamZiane2008}
and stochastic \cite{EwaldPetcuTemam, GlattHoltzTemam1,GlattHoltzTemam2, DebusscheGlattHoltzTemam1,DebusscheGlattHoltzTemamZiane1,
GlattHoltzTemamWang2013, GlattHoltzKukavicaVicolZiane2014} settings.

Note that, in contrast to the Navier-Stokes equations, global existence of strong solutions to the Primitive equations
has been proven in 3D \cite{CaoTiti2007,Kobelkov2007, ZianeKukavica}, but the uniqueness of weak solutions in 2-D remains an
outstanding open problem.  Indeed, for the hydrostatic Navier-Stokes equations, we rely on $H^1$ well-posedness results
\cite{GlattholtzZiane2008,GlattHoltzTemam2} to provide suitable Markovian dynamics associated to \eqref{eq:PEFromalEQ1Momentum}--\eqref{eq:PEFromalEQ1DivFree}.  The existence of invariant measures follows from $H^2$-moment bounds (see \cite{GlattHoltzKukavicaVicolZiane2014} and
 \eqref{eq:H2:bound} below) and the Krylov-Bogoliubov Theorem.  For uniqueness of the invariant measure we invoke another asymptotic coupling argument.
In comparison to the previous example, this argument will invoke the flexibility of Corollary \ref{cor:precise:app}
by proving convergence in the (weaker) $L^2$ topology. The more involved cases of other
boundary conditions, couplings with proxies for density (temperature and salinity), and three space dimensions,
will be pursued in future work.

\subsubsection{The Mathematical Setting}
We begin with some observations about the structure of \eqref{eq:PEFromalEQ1Momentum}--\eqref{eq:PEFromalEQ1DivFree}
when subject to \eqref{eq:HNSE:bcs}.  Firstly notice that, in view of \eqref{eq:PEFromalEQ1DivFree} and \eqref{eq:HNSE:bcs}:
$\int^0_{-h} \partial_x u(x,z) dz = - \int_{-h}^0 \partial_z w(x,z) dz = 0$ for $x \in [0,L]$. It follows that
\begin{align*}
  \int^0_{-h} u(x,z) dz \equiv 0.
\end{align*}
The divergence free condition \eqref{eq:PEFromalEQ1DivFree} coupled with \eqref{eq:HNSE:bcs}
also allows us to write $w$ as a functional of $u$, namely $w(x, z)  =-  \int_{-h}^z \partial_x u(x,\bar{z}) d\bar{z}$.
In the geophysical literature $w$ is referred to as a `diagnostic variable' and we define
\begin{align}
  \wdiag(u)(x,z) = \int_{-h}^z \partial_x u(x,\bar{z}) d\bar{z}
  \label{eq:w:diag:def}
\end{align}

Consider the spaces
\begin{align*}
H = \left\{u\in L^2(\DD):\int_{-h}^{0}u dz \equiv 0\right\} \quad\text{and}\quad V = \left\{u\in H^1(\DD):\int_{-h}^{0}u dz \equiv 0, u|_{\Gamma_l}=0\right\}
\end{align*}
and the projection operator $P_H:L^2(\DD)\rightarrow H$,
\begin{align*}
P_{H}(v)=v -\frac{1}{h}\int_{-h}^{0}v dz.
\end{align*}
As in the previous example we use $| \cdot |$ and $\| \cdot \|$ to denote the $L^2$ and $H^1$ respectively.
We define $A = -P_H \Delta$, and identify its domain as
\begin{align*}
D(A)=\left\{u\in H^2(\DD):\int_{-h}^{0}u dz \equiv 0, u|_{\Gamma_l}=0, \pd{z}u|_{\Gamma_v}=0 \right\}.
\end{align*}
Then for fixed $u_0\in V$, and $\sigma_k \in D(A)$ with finite $d$ (or sufficently fast decay in $\|\sigma_k \|_{H^2}$), the system \eqref{eq:PEFromalEQ1Momentum}--\eqref{eq:HNSE:bcs} possesses a unique solution
\begin{align*}
  u(\cdot) = u(\cdot,u_0)\in L^2(\Omega;C([0,\infty);V) \cap L^2_{loc}([0,\infty); D(A))),
\end{align*}
which depends continuously on $u_0\in V$; see \cite{GlattholtzZiane2008,GlattHoltzTemam2}.
Moreover, the dynamics of \eqref{eq:PEFromalEQ1Momentum}--\eqref{eq:HNSE:bcs}
generate a Feller Markov semigroup $\{P_{t}\}_{t\geq 0}$ on $C_b(V)$.

\subsubsection{A Priori Estimates}

We proceed to establish some estimates on solutions needed for existence and uniqueness of
invariant measures.  The energy estimate is standard.  In view of \eqref{eq:PEFromalEQ1DivFree} the pressure and nonlinear
terms drop and It\={o}'s lemma gives
\begin{align}
  d | u| ^2 + 2\nu \|u\|^2 dt =  |\sigma|^2 dt + \langle u, \sigma \rangle dW.
  \label{eq:energy:prob:bnds:hnse:0}
\end{align}
From the exponential martingale bound, cf. \eqref{eq:exp:mart}, we infer that
\begin{align}
  \Prb \Big(  \sup_{t \geq 0} \   |u(t)|^2 + \nu \int_0^t \| u\|^2 ds  - |\sigma|^2 t - |u_0|^2  \, >\, R  \Big) \leq e^{-\gamma R},
  \label{eq:energy:prob:bnds:hnse}
\end{align}
for every $R > 0$ where $\gamma = \gamma(\nu, |\sigma|^2)$ does not depend on $R$ or the number of forced modes.

We will also require bounds on $\|\pd{z}u\|$, and appeal to the
`vorticity form' of \eqref{eq:PEFromalEQ1Momentum}--\eqref{eq:HNSE:bcs} defined by taking $\partial_z$ of \eqref{eq:PEFromalEQ1Momentum}.
In view of \eqref{eq:Hydrostatic:pres} the pressure $p$ is independent of $z$, and we obtain
\begin{align}
    &d \partial_z u + \partial_z(u \partial_x u + w \partial_z u) dt
                 - \nu \Delta \partial_z udt  =
                  \sum_{k=1}^N \partial_z \sigma_k dW^k.
                  \label{eq:vort:form:Hydro:stat:NSE}
\end{align}
Notice that in view of the boundary conditions \eqref{eq:HNSE:bcs}, $- \int \Delta \partial_z u \partial_z u dx dz =  \| \partial_z u \|^2$, and
moreover we have the cancelation
\begin{align}
   \int\partial_z(u \partial_x u + w \partial_z u) \partial_z u dx dz
   =&    \int( \partial_z u \partial_x u + u \partial_{xz} u +  \partial_z w \partial_z u + w \partial_{zz} u) \partial_z u dx dz \notag \\
   =& \frac{1}{2} \int (  \partial_x u(\partial_z u)^2 +  \partial_z w (\partial_z u)^2 ) dx dz = 0. \label{eq:exact:cancel:pdz:hnse}
\end{align}
Combining these observations we obtain
\begin{align}
      d |\partial_z u|^2 +  2\nu \|\partial_z u\|^2 dt  =  | \partial_z \sigma|^2 +
                  2\langle \partial_z \sigma, \partial_z u \rangle dW,
                  \label{eq:vort:form:Hydro:stat:NSE:energy}
\end{align}
and hence, for every $R > 0$,
\begin{align}
  \Prb \Big(  \sup_{t \geq 0} \   | \partial_z u(t)|^2 + \nu \int_0^t
  \|\partial_z u\|^2 ds  - |\partial_z \sigma|^2 t - |\partial_z
  u_0|^2   \,> \, R  \Big) \leq e^{-\gamma R},
  \label{eq:energy:prob:bnds:hnse:2}
\end{align}
for some $\gamma = \gamma(\nu, |\partial \sigma|) >0$ independent of $R$.

\subsubsection{Existence of Invariant States}

We can now prove existence of invariant measures for \eqref{eq:PEFromalEQ1Momentum}--\eqref{eq:PEFromalEQ1DivFree}
by considering the evolution equation for $\|u\|^2$.   Here we follow an approach similar to \cite{GlattHoltzKukavicaVicolZiane2014}.  By applying the It\={o} lemma to a Galerkin truncation of \eqref{eq:PEFromalEQ1Momentum}--\eqref{eq:HNSE:bcs},
and passing to a limit, we obtain
\begin{align}
 d \| u\|^2  + \nu \| u \|^2_{H^2} dt   \leq C( |\partial_x u|^4 + |\partial_x u|^2 |\partial_z u | \| \partial_z u \|) + \| \sigma\|^2 dt
                 + 2\langle  \sigma,  u \rangle_{H^1} dW.
                 \label{eq:H1:evolution:HNSE}
\end{align}
Here we have used the anisotropic estimate
\begin{align}
  \Big| \int_0^L\int_{-h}^0 \wdiag(v) \partial_z u_1 u_2 dz dx\Big|
  \leq& h^{1/2} \int_0^L  | \partial_x v |_{L^2_z} |\partial_z u_1|_{L^2_z}  | u_2 |_{L^2_z} dx
  \leq h^{1/2}  |u_2| |  \partial_x v |  \Big( \sup_{x \in [0,L]}  \int_{-h}^0 (\partial_z u_1)^2 dz \Big)^{1/2}
  \notag\\
  \leq& h^{1/2}  |u_2| | \partial_x v | \! \Big( \sup_{x \in [0,L]} \int_0^x \!\! \partial_{\bar{x}} \! \int_{-h}^0 (\partial_z u_1)^2 dz  d\bar{x}\Big)^{1/2}
  \notag\\
  \leq& 2h^{1/2} |u_2| | \partial_x v | |\partial_z u_1 |^{1/2} \| \partial_z u_1\|^{1/2}
  \label{eq:aniso:est:ex:1}
\end{align}
for all suitably regular $v, u_1, u_2$.  With \eqref{eq:aniso:est:ex:1} and the Sobolev embedding of $H^{1/3}$ into $L^3$
in dimension $2$ we infer that
\begin{align*}
 \left| \int (u \partial_x u + w \partial_z u) \partial_{xx} u dx dz \right|
 &\leq C( |\partial_x u|_{L^3}^3 + |\partial_x u| \|\partial_x u \| |\partial_z u |^{1/2} \| \partial_z u \|^{1/2})\\
 &\leq C( |\partial_x u|^2 \| \partial_x u\| + |\partial_x u| \|\partial_x u \| |\partial_z u |^{1/2} \| \partial_z u \|^{1/2})
 \\
 &\leq \nu \|u\|^2_{H^2} + C(|\partial_x u|^4 + |\partial_x u|^2|\partial_z u | \| \partial_z u \|),
\end{align*}
and from here the derivation of \eqref{eq:H1:evolution:HNSE} is straightforward.
From \eqref{eq:H1:evolution:HNSE} we now compute $\log(1 + \| u\|^2)$ and observe
\begin{align*}
  d \log(1 + \| u\|^2)  + \nu  \frac{\| u \|^2_{H^2}}{1 + \| u\|^2}dt \leq C( \| u\|^2 + |\partial_z u | \| \partial_z u \|) dt
  + \| \sigma\|^2 dt
                 + 2\frac{\langle  \sigma,  u \rangle_{H^1}}{1 + \| u\|^2} dW.
\end{align*}
Hence from this bound and  \eqref{eq:energy:prob:bnds:hnse:0}, \eqref{eq:vort:form:Hydro:stat:NSE:energy}
we infer
\begin{align}
\label{eq:H2:bound}
  \int_0^T  \E \| u  \|_{H^2} ds 
  \leq  \frac{1}{2} \E \int_0^T \Big( 1 +  \| u\|^2 + \frac{ \| u  \|_{H^2}^2}{1 + \| u\|^2} \Big)ds
  \leq C (T + 1),
\end{align}
for a constant $C = C(\|u_0\|^2, \|\sigma\|^2, \nu)$ independent of $T > 0$.  The existence of invariant measures associated to \eqref{eq:PEFromalEQ1Momentum}--\eqref{eq:PEFromalEQ1DivFree} now follows by applying the Krylov-Bogolyubov Theorem.

\subsubsection{Asymptotic Coupling Arguments}
Similar to the previous example and again following Section~\ref{sec:recipeC}, we fix any $u_0, \tilde{u}_0 \in V$
and consider $u$ a solution of \eqref{eq:PEFromalEQ1Momentum}--\eqref{eq:PEFromalEQ1DivFree} starting from
$u_0$ and $\tilde{u}$ solving the same system with an additional control $G$ given as
\begin{align*}
  \lambda \indFn{\{\tau_K > t\}} P_N (u - \tilde{u})  dt, \quad \textrm{ with } \quad
     \tau_K := \inf_{t \geq 0} \left\{ \int_0^t \|P_N (u - \tilde{u})\|^2  ds \geq K \right\},
\end{align*}
and starting from $\tilde{u}_0$; cf. \eqref{eq:nse:shift}.
Once again, the parameters $\lambda , K > 0$ and $N$ will be specified below.
As above $\tilde{u}$ is subject to a  Girsonov shift of the form $\sigma^{-1}G$
and Theorem~\ref{thm:GirsonovRedux} applies.

Subtracting $\tilde{u}$ from $u$ and taking
$v = u - \tilde{u}$, $q = p - \tilde{p}$ we obtain
\begin{align*}
  \partial_t v - \nu \Delta v + \lambda P_N v dt  = -
                 \tilde{u} \partial_x v - \wdiag(\tilde{u}) \partial_z v
                 - v \partial_x u - \wdiag(v) \partial_z u
                 - \partial_x q.
\end{align*}
It follows, as in \eqref{eq:inv:pon:app} for suitably large $N$, that on $[0, \tau_K]$
\begin{align}
  \frac{1}{2} \frac{d}{dt} | v |^2 + \frac{\nu}{2} \| v\|^2 + \lambda |v|^2
  \leq \left| \int (v \partial_x u + \wdiag(v) \partial_z u) v dx dz \right|
  \leq C  |v| \| v\| (\|u\| + | \partial_z u|^{1/2}\| \partial_z u\|^{1/2}),
  \label{eq:diff:eq:diff:HNSE}
\end{align}
where we have applied the anisotropic estimate \eqref{eq:aniso:est:ex:1} in the last line.
On the interval $[0, \tau_K]$ this gives
\begin{align}
   | v(t) |^2  \leq \exp\left( 2 \lambda t - C\int_0^t  (\|u\|^2 + | \partial_z u| \| \partial_z u\| )ds \right) |v(0)|^2,
   \label{eq:hnse:decay:setup}
\end{align}
where, to emphasize, the constant $C$ is independent of $K$.
Thus, following the strategy of Section \ref{sec:NSE:Channel} above, we can combine \eqref{eq:energy:prob:bnds:hnse}, \eqref{eq:energy:prob:bnds:hnse:2} and
\eqref{eq:hnse:decay:setup} to apply Corollary~\ref{cor:precise:app} with $H_0=H^1(\DD)$ and $\tilde{\rho}$ corresponding to the $L^2(\DD)$-topology.
 We have proven the following result:
\begin{proposition}
  \label{prop:hnse:unique}
  For every $\nu > 0$ there exists $N = N(\nu, | \sigma |^2, |\partial_z \sigma|^2)$ such
  that if $\mathrm{Range}(\sigma) \supset V_N=P_N V$ then
  \eqref{eq:PEFromalEQ1Momentum}--\eqref{eq:PEFromalEQ1DivFree} has a unique ergodic invariant measure.
\end{proposition}

\subsection{The Fractionally Dissipative Euler Model}
\label{sec:Frac:disp:Euler}

Our next example considers the fractionally dissipative Euler equations introduced in \cite{ConstantinGlattHoltzVicol2013}.
This system takes the form
\begin{align}
  d \Vort + (\Lambda^\gamma \Vort + \bfU \cdot \nabla \Vort)dt = \sum_{k =1}^d \sigma_k dW^k, \quad \bfU = \mathcal{K} \ast \Vort,
  \label{eq:frac:SNSE}
\end{align}
for any unknown vorticity field $\Vort$.  Here $\Lambda^\gamma = (- \Delta)^{\gamma/2}$ is the fractional Laplacian which we consider for
\emph{any}  $\gamma \in (0,2]$,
$\mathcal{K}$ is the Biot-Savart kernel, so that $\nabla^\perp \cdot \bfU = \Vort$ and $\nabla \cdot \bfU = 0$, and we suppose that \eqref{eq:frac:SNSE}
is posed on the periodic box $\TT^2 = [-\pi , \pi]^2$.  Conditions on the forced directions $\sigma_k$ will be specified below.

In \cite{ConstantinGlattHoltzVicol2013} it was demonstrated that with
``effectively elliptic'' forcing, the system \eqref{eq:frac:SNSE}
possesses a unique ergodic invariant measure, and in the course of the proof, significant effort was made to establish arbitrary order polynomial moment bounds in high order Sobolev spaces ($H^r$ for any $r > 2$).  These bounds are interesting and hold significance for questions regarding the rate of convergence to the invariant measure.  However, we
show here that much less effort is required if one simply wishes to prove existence and uniqueness of the invariant measure.
As in the last example, the argument is significantly simplified by invoking Corollary \ref{cor:precise:app} and proving convergence in the $L^2$-topology.

\subsubsection{Mathematical Preliminaries}
We consider \eqref{eq:frac:SNSE} in its velocity formulation
\begin{align}
  d \bfU + (\Lambda^\gamma \bfU + \bfU \cdot \nabla \bfU + \nabla\pres)dt = \sum_{k =1}^d \rho_k dW^k, \quad \nabla \cdot \bfU = 0.
  \label{eq:frac:SNSE:vel}
\end{align}
Here the unknowns are the velocity field $\bfU = (u_1, u_2)$ and pressure $\pres$ both posed on the periodic box $\TT^2$.  One
may show using a Galerkin regularization argument, that for any $r >2$ and $\bfU_0\in H^r$, there exists a unique $\bfU = \bfU(\spcdot, \bfU_0)$ solving
\eqref{eq:frac:SNSE:vel}, with $\bfU \in C([0,\infty), H^r)$.  As in \cite{ConstantinGlattHoltzVicol2013} we may infer that for any such $r >2$ and any $t > 0$, $\bfU(t, \bfU_0^n)  \to \bfU(t, \bfU_0)$
almost surely in $H^r$ whenever $\bfU_0^n \to \bfU_0$ in $H^r$.  It follows that the
transition function $P_t(\bfU_0, A) = \Prb( \bfU(t, \bfU_0) \in A)$, $\bfU_0 \in H^r$,  $A \in \mathcal{B}(H^r)$
defines a Feller Markovian semigroup.

To prove the existence of an invariant measure for \eqref{eq:frac:SNSE} we argue as follows.  Applying $\Lambda^r$ for any $r > 2$ to \eqref{eq:frac:SNSE:vel} and then integrating
we find that
\begin{align}
  d \| \bfU \|_{H^r}^2 + 2\| \bfU \|_{H^{r+\frac{\gamma}{2}}}^2 dt = - 2\int_{\TT^2} ( \Lambda^r( \bfU \cdot \nabla \bfU) - \bfU \cdot \nabla \Lambda^r \bfU) \Lambda^r \bfU dx
      + \| \rho \|_{H^r}^2 dt + 2\langle \rho, \bfU \rangle_{H^r} dW,
      \label{eq:Hr:evol}
\end{align}
where we have used that $\bfU$ is divergence free to eliminate the pressure and
rewrite the nonlinear terms.   In order to estimate these nonlinear terms we recall the Kenig-Ponce-Vega commutator estimate
\begin{align}
  \| \Lambda^{s} (f \cdot \nabla g) - f \cdot \nabla \Lambda^{s} g\|_{L^p} \leq C( \| \Lambda^{s} f\|_{L^{q_1}}  \| \nabla g\|_{L^{r_1}}  +\| \nabla f\|_{L^{q_2}}  \| \Lambda^{s} g\|_{L^{r_2}}),
  \label{eq:KPV:comm:est}
\end{align}
valid for any suitably regular $f,g$,  $s > 1$ and any trios $p, q_i, r_i$ with $1 < p, q_i, r_i < \infty$ and $p^{-1} = q^{-1}_j + r^{-1}_j$, $j =1,2$; see \cite{MuscaluSchlag13}.  With this bound we obtain
\begin{align}
 \left| \int_{\TT^2} ( \Lambda^r( \bfU \cdot \nabla \bfU) - \bfU \cdot \nabla \Lambda^r \bfU) \Lambda^r \bfU dx \right|
      \leq&C \| \nabla \bfU\|_{L^{4/\delta}}  \| \Lambda^{r} \bfU\|_{L^{4/(2- \delta)}}  |\Lambda^r \bfU |,
        \label{eq:nlt:frac:e:0}
\end{align}
valid for any $\delta \in [0,1)$.     Next recall the Gagliardo-Nirenberg interpolation inequality
\begin{align*}
  \| \Lambda^{\alpha} f\|_{L^p} \leq C \| f \|^{\theta}_{L^q}  \| \Lambda^{\beta} f \|^{1 - \theta}_{L^m},
\end{align*}
which holds for any $1 < p, q, m \leq \infty$, $0 < \alpha < \beta < \infty$ such that
\begin{align*}
  \frac{1}{p} = \frac{\theta}{q} + \frac{1 - \theta}{m}, \quad \textrm{ where } \quad \theta = 1 - \frac{\alpha}{\beta}.
\end{align*}
Taking $p = 4/(2- \delta)$, $m = 2$, $\alpha = r -1$ and $\beta = r -1 + \frac{\gamma}{2}$ in this inequality we
infer that for every $\delta <  \frac{2 \gamma}{\gamma + 2r -2},$
\begin{align*}
   \| \Lambda^{r} \bfU\|_{L^{4/(2- \delta)}}  \leq C \| \nabla \bfU \|_{L^{q}}^{\frac{\gamma}{2r -2 + \gamma}}|  \Lambda^{r+\frac{\gamma}{2}} \bfU |^{\frac{2r - 2}{2r-2+\gamma}}
   \quad \textrm{ with } \quad q = \frac{4\gamma}{ 2\gamma - \delta(\gamma + 2r - 2)}.
\end{align*}
Hence by choosing $\delta = \frac{\gamma}{\gamma + 2r -2}$ and combining this bound with \eqref{eq:nlt:frac:e:0}
we find
\begin{align}
 \left| \int_{\TT^2} ( \Lambda^r( \bfU \cdot \nabla \bfU) - \bfU \cdot \nabla \Lambda^r \bfU) \Lambda^r \bfU dx \right|
      \leq& C \| \Vort \|_{L^{4/\delta}} \| \Vort \|_{L^{4}}^{\frac{\gamma}{2r-2 + \gamma}} \| \bfU\|_{H^{r + \frac{\gamma}{2}}}^{\frac{2r - 2}{2r-2+\gamma}} \| \bfU \|_{H^r}
      \notag\\
      \leq& C \| \Vort \|_{L^{4/\delta}}^{2} \| \bfU \|_{H^r}^{2\left(\frac{2r-2 + \gamma }{2r-2 + 2\gamma }\right)}
      + \| \bfU\|_{H^{r + \frac{\gamma}{2}}}^2.
      \label{eq:nlt:frac:e:1}
\end{align}
With this estimate in mind we now compute a differential for $(1 + \| \bfU \|_{H^r}^2)^\kappa$ with $\kappa \in (0,1)$.   This yields
\begin{align*}
 d (1 + \| \bfU \|_{H^r}^2)^\kappa + \kappa \| \bfU \|_{H^{r+\frac{\gamma}{2}}}^2 &(1 + \| \bfU \|_{H^r}^2)^{\kappa-1} dt
     \leq  C \| \Vort \|_{L^{4/\delta}}^{2} \| \bfU \|_{H^r}^{2\left(\frac{2r-2 + \gamma }{2r-2 + 2\gamma }\right)} (1 + \| \bfU \|_{H^r}^2)^{\kappa-1} dt\\
      &+ \kappa \| \rho \|_{H^r}^2 (1 + \| \bfU \|_{H^r}^2)^{\kappa-1} dt + 2\kappa\langle \rho, \bfU \rangle_{H^r} (1 + \| \bfU \|_{H^r}^2)^{\kappa-1} dW.
\end{align*}
By taking $\kappa = \frac{\gamma}{2r -2 + 2 \gamma}$ we infer
\begin{align}
  \int_0^t \E  \| \bfU \|_{H^{r+\frac{\gamma}{2}}}^\kappa ds \leq C \left( (1 + \| \bfU_0 \|_{H^r}^2)^\kappa +\int_0^t  (\| \rho \|_{H^r}^2  +  \E\| \Vort \|_{L^{4/\delta}}^{2} + 1) ds\right),
  \label{eq:KB:bnd:frack:oiler}
\end{align}
for each $t \geq 0$ where the constant $C = C(\gamma,r)$ is independent of $\bfU_0$ and $t$.  The existence of an invariant measure
now follows once we establish a suitable bound on $\Vort$ in $L^p(\TT^2)$ for any $p \geq 2$.

To this end, we next observe that from \eqref{eq:frac:SNSE} we have for any $p \geq 2$,
\begin{align*}
  d \| \Vort \|_{L^p}^p + p \int_{\TT^2} \Lambda^\gamma \Vort \Vort^{p-1} dx dt
  	= \frac{p(p-1)}{2} \sum_{k =1}^d \int \sigma_k^2 \Vort^{p-2} dx dt
	+ p \sum_{k =1}^d \int \sigma_k \Vort^{p-1} dx dW^k.
\end{align*}
Recalling the nonlinear Poincar\'e inequality from \cite{ConstantinGlattHoltzVicol2013}
\begin{align*}
  p \int_{\TT^2} \Lambda^\gamma \Vort \Vort^{p-1} dx \geq \frac{1}{C_\gamma} \| \Vort \|_{L^p}^p,
\end{align*}
where the constant $C_\gamma$ depends only on $\gamma > 0$, we infer
\begin{align}
  d \| \Vort \|_{L^p}^p + \frac{1}{C_\gamma} \| \Vort \|_{L^p}^p dt
        \leq \frac{p(p-1)}{2} \| \sigma \|_{L^p}^2 \| \Vort \|_{L^p}^{p-2} dt
	+ p \sum_{k =1}^d \int \sigma_k \Vort^{p-1} dx dW^k.
	\label{eq:Lp:bnd:frak:oiler}
\end{align}
The existence of an (ergodic) invariant measure follows immediately by combining \eqref{eq:Lp:bnd:frak:oiler} with \eqref{eq:KB:bnd:frack:oiler}.

\subsubsection{Asymptotic Coupling Arguments}
Fix any $\bfU_0, \tbfU_0 \in H^r$ and let $\bfU = \bfU(\spcdot, \bfU_0)$ be the corresponding solution of \eqref{eq:frac:SNSE} while
we suppose that $\tbfU$ solves
\begin{align*}
  d \tbfU + (\Lambda^\gamma \tbfU - \indFn{\tau_K > t} \lambda P_N (\bfU - \tbfU) + \tbfU \cdot \nabla \tbfU + \nabla \tpres)dt = \sum_{k =1}^d \sigma_k dW,
  \quad \tbfU(0) = \tbfU_0,
\end{align*}
where
\begin{align*}
  \tau_K := \inf_{t \geq 0}  \left\{   \int_0^t |P_N  (\bfU - \tbfU)|^2 ds \geq K \right\}.
\end{align*}
The parameters $K, \lambda > 0$ are to be determined presently.  It is easy to see from Theorem~\ref{thm:GirsonovRedux} that, for any choice of $\lambda, K > 0$,
 the law of $\tbfU$ is absolutely continuous with
respect to the solution $\bfU( \cdot , \tbfU_0)$ of \eqref{eq:frac:SNSE} corresponding to $\tbfU_0$.     As previous examples, unique ergodicity follows
from Corollary~\ref{cor:precise:app} once we can find some  $\lambda, K > 0$ (where $K$ may depend on $\bfU_0, \tbfU_0$)
such that $\bfU(t) - \tbfU(t) \to 0$ in $L^2(\TT^2)$ on a set of non-trivial measure.

Take $\bfV = \bfU - \tbfU$ and $\dpres = \pres - \tpres$. We find
\begin{align*}
 \partial_t \bfV + \Lambda^\gamma \bfV + \indFn{\tau_K > t} \lambda P_N \bfV + \bfV \cdot \nabla \bfU + \tbfU \cdot \nabla \bfV + \nabla \dpres = 0.
\end{align*}
Hence, using that $\tilde{\bfU}$ is divergence free and the generalized Poincar\'e inequality (similarly to \eqref{eq:inv:pon:app} above) we have
\begin{align*}
  \frac{d}{dt} | \bfV |^2 + 2 \lambda | \bfV |^2  + \| \bfV\|^2_{H^{\gamma/2}}
    \leq 2 \left| \int \bfV \cdot \nabla \bfU \cdot \bfV dx \right|
\end{align*}
for every $t \in [0, \tau_K]$.  Here $\lambda = \lambda(N, \gamma)$ can be chosen as large
as desired by decreeing the space $H_N$ spanned by the forced modes
to be commensurately big.
By choosing $p =p(\gamma) >0$ sufficiently large we infer
\begin{align*}
  2 \left| \int \bfV \cdot \nabla \bfU \cdot \bfV dx \right| \leq C \| \Vort \|_{L^p} | \bfV | \| \bfV \|_{H^{\gamma/2}}
\end{align*}
and hence the bound
\begin{align*}
 \frac{d}{dt} | \bfV |^2 + (2 \lambda  - C \| \Vort \|_{L^p}^2 )| \bfV |^2
    \leq 0
\end{align*}
holds on $[0, \tau_K]$.  Gr\"onwall's inequality implies that, on this same interval $[0, \tau_K]$,
\begin{align}
  | \bfV (t) |^2
    \leq | \bfV(0)|^2 \exp\big( - 2 \lambda t + C\int_0^t  \| \Vort \|_{L^p}^2 ds \big).
        \label{eq:diff:ineq:bfV:good:int}
\end{align}
Here we again emphasize that the constant $C$ appearing in the exponential depends only
on universal quantities and in particular is independent of $K$ in the definition of $\tau_K$.

To finally infer the desired contraction from \eqref{eq:diff:ineq:bfV:good:int}
we thus need a further bound on the $L^p$ norms of $\Vort$.  For this we
compute $d (1+ \| \Vort \|_{L^p}^p)^{2/p}$ which with \eqref{eq:Lp:bnd:frak:oiler} yields
\begin{align}
  d (1+ \| \Vort \|_{L^p}^p)^{2/p} + \frac{2}{pC_\gamma} \| \Vort \|_{L^p}^2 dt
  	\leq (p-1) \| \sigma \|_{L^p}^2 dt
	+ 2 (1+ \| \Vort \|_{L^p}^p)^{\frac{2}{p}-1}   \sum_{k =1}^d \int \sigma_k \Vort^{p-1} dx dW^k.
	\label{eq:Vort:p:est:1}
\end{align}
Observe that the martingale term in the inequality above has a quadratic variation which can be estimated as
\begin{align}
4 \int_0^t(1+ \| \Vort \|_{L^p}^p)^{\frac{4}{p}-2}   \sum_{k =1}^d \Big(\int \sigma_k \Vort^{p-1} dx \Big)^2 ds
 \leq& 4 \int_0^t(1+ \| \Vort \|_{L^p}^p)^{\frac{4}{p}-2}    \Big(\int  \Big(\sum_{k =1}^d \sigma_k^2 \Big)^{1/2} \Vort^{p-1} dx \Big)^2 ds
 \notag\\
  \leq& 4 \|\sigma\|^2_{L^p} \int_0^t (1 + \| \Vort \|_{L^p}^p)^{2/p} ds.
  	\label{eq:Vort:p:est:2}
\end{align}
By now combining \eqref{eq:Vort:p:est:1}, \eqref{eq:Vort:p:est:2} we infer from exponential martingale bounds, \eqref{eq:exp:mart}, that
\begin{align}
  \Prb\left( \inf_{t \geq 0}
        \ \frac{1}{pC_\gamma} \int_0^t \| \Vort \|_{L^p}^2 dt - (p + 2^{2/p+2}) \| \sigma \|_{L^p}^2 t
     \    \geq R \right) \leq e^{-\alpha R},
        \label{eq:exp:MG:bnd:lp}
\end{align}
for every $R \geq 0$ where $\alpha = \alpha( \|\sigma\|_{L^p}, p,  \gamma)$ is independent of $R$ and
does not depend on the number of forced modes but only on the norm of $\|\sigma\|_{L^p}$.

Combining \eqref{eq:diff:ineq:bfV:good:int} and \eqref{eq:exp:MG:bnd:lp} and arguing as in the previous examples
we infer that, for an appropriate choice of $K >0$,  $|v(t)| \to 0$ on a set of non-trivial measure.
In summary we have proven the following:
\begin{proposition}
 The system \eqref{eq:frac:SNSE} possesses an ergodic invariant measure. When $N = N( \|\sigma\|_{H^r}, \gamma)$
 is sufficiently large and $\mathrm{Range}(\sigma) \supset P_N H^r$ this invariant measure is unique.
\end{proposition}

\subsection{The Damped Stochastically Forced Euler-Voigt Model}
\label{sec:e:voigt}

The next system that we will consider is an inviscid `Voigt-type' regularization (see e.g.
\cite{Oskolkov1977} and further references below)  of the damped stochastic
Euler equations. This example is significant as, in contrast to the previous equations,
it illustrates a case for which the existence and uniqueness of invariant measures
can be demonstrated in the absence of a parabolic regularization mechanism.
In fact both the questions of the existence and the uniqueness of the invariant measure
leads to interesting new twists in the analysis in comparison to the previous examples.
For the question of existence we make use of an inviscid limit procedure
along with an abstract result presented in Corollary~\ref{cor:ex:c} in Appendix~\ref{sec:ex:IM} below.

The governing equations read
\begin{align}
  d \bfU + (\gamma \bfU + \bfU_\alpha \cdot \nabla \bfU_\alpha + \nabla p) dt = \sum_{k =1}^d \sigma_k dW^k,
  \quad \bfU(0) = \bfU_0,
  \label{eq:e:v:1}
\end{align}
for some $\gamma > 0$ with the unknown vector field $\bfU$ subject to the divergence-free condition $\nabla \cdot \bfU =0$
and where the non-linear terms are subject to an $\alpha$ degree
regularization
\begin{align}
  (- \Delta)^{\alpha/2} \bfU_\alpha = \Lambda^{\alpha} \bfU_\alpha = \bfU.
    \label{eq:e:v:2}
\end{align}
We suppose that \eqref{eq:e:v:1} evolves on the periodic box $\TT^n$ where $n = 2,3$.   To streamline our presentation and in view of the
fact that damping terms are more natural for two dimensional flows, our main focus
will be on the case $n = 2$.  Here the assumed degree of regularization $\alpha$  in \eqref{eq:e:v:2} is greater $2/3$.  This lower bound
is a strict inequality for the question of uniqueness.     Note however that the case $n =3$ can be addressed by a similar approach when
we suppose that $\alpha \geq 2$.  See Remark~\ref{rmk:3d:case} at the conclusion of this section for further details.

There is a vast literature around regularizations (or mollifications) of the nonlinear terms in the Navier-Stokes and Euler equations.
In fact, it is notable that such a regularization procedure was the basis for the first existence results for weak solutions dating back to the seminal work of
Leray, \cite{Leray1934}.   In the more recent literature a variety of related systems explore this theme in the context of turbulence
closure models, viscoelastic and non-newtonian fluids and a variety of other applications.  See, for example,
\cite{Oskolkov1977, FoiasHolmTiti2002, MarsdenShkoller2001,  MarsdenShkoller2003,CheskidovHolmOlsonTiti2005,
LaytonLewandowski2006, CaoLunasinTiti2006, LariosTiti2009, KalantarovTiti2009, DiMolfettaKrstlulovicBrachet2015} and
numerous containing references.

\subsubsection{A Priori Estimates}

We begin by illustrating some a-priori energy estimate for \eqref{eq:e:v:1}--\eqref{eq:e:v:2} which will guide us
in the sequel. Notice that if we apply $\Lambda^{-\alpha/2}$ to \eqref{eq:e:v:1} we obtain from the It\={o} lemma that
\begin{align}
 d \| \Lambda^{-\alpha/2} \bfU \|^2+  2\gamma  \| \Lambda^{-\alpha/2} \bfU \|^2 dt
 	= \|\Lambda^{-\alpha/2} \sigma\|^2 + 2 \langle \Lambda^{-\alpha/2} \sigma,  \Lambda^{-\alpha/2} \bfU \rangle dW,
	\label{eq:ap:1}
\end{align}
and hence exponential martingale bounds imply
\begin{align}
   \Prb\left( \sup_{t \geq 0} \Big(\| \Lambda^{-\alpha/2} \bfU (t)\|^2+  \gamma \int_0^t \| \Lambda^{-\alpha/2} \bfU \|^2 ds -
   \|\Lambda^{-\alpha/2} \sigma\|^2 t  + \|\Lambda^{-\alpha/2} \bfU_0\|^2 \Big) \geq K\right) \leq \exp(- c K),
   \label{eq:exp:mar:1}
\end{align}
for each $K > 0$ and some $c = c(\|\Lambda^{-\alpha/2} \sigma\|^2, \gamma)$ independent of $K$.

Next observe that, by taking $\xi = \mbox{curl} \,\bfU$, $\rho = \mbox{curl} \, \sigma$, we obtain the vorticity formulation of \eqref{eq:e:v:1}
\begin{align*}
  d \xi + (\gamma \xi + \bfU_\alpha \cdot \nabla \xi_\alpha - \xi_\alpha \cdot \nabla \bfU_\alpha) dt = \sum_{k =1}^N \rho_k dW^k,
\end{align*}
In $n =2$, our main concern here, the `vortex stretching term' $\xi_\alpha \cdot \nabla \bfU_\alpha$ is absent and we obtain
\begin{align}
  d \xi + (\gamma \xi + \bfU_\alpha \cdot \nabla \xi_\alpha) dt = \sum_{k =1}^N \rho_k dW^k,
  \label{eq:e:v:vort}
\end{align}
which, in this two dimensional case, implies
\begin{align}
 d \| \Lambda^{-\alpha/2} \xi \|^2+  2\gamma  \| \Lambda^{-\alpha/2} \xi \|^2 dt
 	= \|\Lambda^{-\alpha/2} \rho \|^2 + 2 \langle \Lambda^{-\alpha/2} \rho,  \Lambda^{-\alpha/2} \xi \rangle dW,
	  \label{eq:ap:2}
\end{align}
and hence yields
\begin{align}
   \Prb\left( \sup_{t \geq 0} \Big(\| \Lambda^{-\alpha/2} \xi(t)\|^2+  \gamma \int_0^t \| \Lambda^{-\alpha/2} \xi \|^2 ds -
   \|\Lambda^{-\alpha/2} \rho\|^2 t  + \|\Lambda^{-\alpha/2} \xi_0\|^2 \Big) \geq K\right) \leq \exp(- c K),
      \label{eq:exp:mar:2}
\end{align}
for each $K > 0$ and some $c = c(\|\Lambda^{-\alpha/2} \rho \|^2, \gamma)$ independent of $K$.
From \eqref{eq:ap:2} we can further prove that for $\eta = \eta(\|\Lambda^{-\alpha/2} \rho\|^2, \gamma)$
\begin{align}
   \E \exp( \eta  \| \Lambda^{-\alpha/2} \xi (t)\|^2) \leq  \exp( \eta ( \gamma^{-1}\| \Lambda^{-\alpha/2} \rho\|^2 + e^{-\gamma t /2 }  \|\Lambda^{-\alpha/2} \xi_0\|^2))
   \label{eq:exp:mar:3}
\end{align}
and we also have that
\begin{align}
   \E \exp \Big( \eta \gamma  \int_0^t \| \Lambda^{-\alpha/2} \xi (t)\|^2 ds \Big) \leq  \exp( \| \Lambda^{-\alpha/2} \rho\|^2 t +  \|\Lambda^{-\alpha/2} \xi_0\|^2).
      \label{eq:exp:mar:4}
\end{align}
Note that the constant $\eta$ appearing in \eqref{eq:exp:mar:3}, \eqref{eq:exp:mar:4} may be taken to be less than $1$.

Suppose that $\bfU$,  $\tbfU$ solve both \eqref{eq:e:v:1}--\eqref{eq:e:v:2} and take $\bfV = \bfU - \tbfU$ which satisfies
\begin{align}
  \pd{t} \bfV +  \gamma \bfV + \bfV_\alpha \cdot \nabla \bfU_\alpha +   \tilde{\bfU}_\alpha \cdot \nabla \bfV_\alpha +  \nabla q = 0, \quad \bfV(0) = \bfU(0) - \tbfU(0)
\end{align}
with $q$ the difference of the pressures.
We immediately infer that
\begin{align}
  \frac{1}{2} \frac{d}{dt} \|\Lambda^{-\alpha/2} \bfV\|^2 +  \gamma \|\Lambda^{-\alpha/2}  \bfV\|^2 = - \int  \bfV_\alpha \cdot \nabla \bfU_\alpha \cdot   \bfV_\alpha dx.
  \label{eq:ap:3}
\end{align}
When $\alpha \geq 2/3$ we have
$\frac{1}{3} \geq \frac{1}{2} - \frac{\alpha}{4}$ and hence (in $n = 2$)
with the Sobolev imbedding of $H^{\alpha/2} \subset L^3$ and basic properties of the Biot-Savart kernel we infer
\begin{align}
 \left| \int  \bfV_\alpha \cdot \nabla \bfU_\alpha \cdot   \bfV_\alpha dx  \right|
 &\leq \|\bfV_\alpha\|_{L^3}^2 \|\nabla \bfU_\alpha\|_{L^3}
 \leq C\| \Lambda^{-\alpha} \bfV \|_{H^{\alpha/2}}^2  \|\Lambda^{-\alpha} \nabla \bfU\|_{H^{\alpha/2}}
 \leq C\| \Lambda^{-\alpha} \bfV \|_{H^{\alpha/2}}^2  \|\Lambda^{-\alpha} \xi \|_{H^{\alpha/2}} \notag\\
 &\leq C \| \Lambda^{-\alpha/2} \bfV \|^2  \|\Lambda^{-\alpha/2} \xi \|.
 \label{eq:2d:bnd:1}
\end{align}

\subsubsection{Existence and Uniqueness of Solutions and  Markov Semigroup}
With these observation in hand we turn now to address the well-possedness for \eqref{eq:e:v:1}--\eqref{eq:e:v:2}.
The a priori bounds \eqref{eq:ap:1}, \eqref{eq:ap:2} in combination with   \eqref{eq:ap:3}--\eqref{eq:2d:bnd:1} are the basis of:
\begin{proposition}\label{prop:e:v:1}
  Assume that $\alpha \geq 2/3$ and consider \eqref{eq:e:v:1} in the case $n = 2$.  Then, for all $\bfU_0 \in H^{1- \alpha/2}$, there exists a unique
  \begin{align*}
   	\bfU \in L^2(\Omega; L^\infty_{loc}([0, \infty); H^{1- \alpha/2}))
  \end{align*}
  evolving continuously in $L^2$ which is an adapted, pathwise solution of \eqref{eq:e:v:1}.  Taking $\bfU(t,\bfU_0)$ as the unique solution associated
  to a given $\bfU_0 \in H^{1- \alpha/2}$ we have that
  \begin{align}
     \bfU(t, \bfU_0^n) \to      \bfU(t, \bfU_0)
    \textrm{ almost surely in the $H^{-\alpha/2}$ topology}
    \label{eq:w:top:conv}
  \end{align}
  for any sequence $\{\bfU_0^n\}_{n \geq 1} \subset H^{1- \alpha/2}$
  converging in $H^{-\alpha/2}$.
  \end{proposition}
  \noindent The existence of solutions in this class may be established with a standard Fado-Galerkin procedure.   We omit further details.

Given Proposition~\ref{prop:e:v:1} we may thus define the Markov transition kernels $\{P_t\}_{t \geq 0}$ associated to \eqref{eq:e:v:1}--\eqref{eq:e:v:2}
as
  \begin{align*}
     P_t(\bfU_0, A) = \Prb(\bfU(t, \bfU_0) \in  A)
  \end{align*}
  These kernels are Feller in  $H^{-\alpha/2}$ namely, given any $\phi \in C_b(H^{-\alpha/2})$, $t \geq 0$,
$P_t \phi \in C_b(H^{-\alpha/2})$.

\subsubsection{The Existence of an Invariant Measure (n =2)}

To prove the existence of an invariant measure we make use of the abstract results in Appendix~\ref{sec:ex:IM}.
In the present concrete setting we take $\spV = H^{1 - \alpha/2}$ and $\spH = H^{-\alpha/2}$.  It is easy to see that (by for example taking $\rho_n$
to be the projection onto $H_n$, the span of the first $n$ elements of a sinusoidal basis) these spaces satisfy the conditions
imposed on $\spV, \spH$ in the Appendix.      Notice moreover that, as we identified in Proposition~\ref{prop:e:v:1} and the surrounding commentary,
the Markov transition kernel associated to  \eqref{eq:e:v:1}--\eqref{eq:e:v:2} is defined on $V$ and is readily seen to be $H$-Feller.

In order to apply Corollary~\ref{cor:ex:c} and hence infer the existence of invariant states we
we now consider,  for each $\epsilon > 0$, the viscous regularizations of \eqref{eq:e:v:1} given as
\begin{align}
  d \bfUe + (\gamma \bfUe - \epsilon \Delta \bfUe + \bfUe_\alpha \cdot \nabla \bfUe_\alpha + \nabla p) dt = \sum_{k =1}^N \sigma_k dW^k, \quad \nabla \cdot \bfUe = 0,
  \quad \bfUe(0) = \bfU_0.
  \label{eq:e:v:vis:1}
\end{align}
As above \eqref{eq:e:v:vis:1} has an associated vorticity form
\begin{align}
  d \qe + (\gamma \qe - \epsilon \Delta \qe + \bfUe_\alpha \cdot \nabla \qe_\alpha) dt = \sum_{k =1}^N \rho_k dW^k.
  \label{eq:e:v:vis:2}
\end{align}
For the same reasons as  \eqref{eq:e:v:1}--\eqref{eq:e:v:2} these equations define
a collections of Markov kernels $\{P^\epsilon_t\}_{t \geq 0}$ for each $\epsilon > 0$ on $\spV = H^{1 - \alpha/2}$.

From the It\={o} lemma we obtain an evolution like  \eqref{eq:ap:2} for $\|\Lambda^{-\alpha/2} \qe \|^2$ but which
has the additional viscous term $2\epsilon \|\nabla  \Lambda^{-\alpha/2} \qe \|^2dt$.  We thus
obtain, for any $t > 0$
\begin{align}
   \epsilon \E \int_0^t  \|\nabla \Lambda^{-\alpha/2} \qe\|^2 ds
   = \epsilon \E \int_0^t  \|  \Lambda^{2-\alpha/2}  \bfUe \|^2 ds
   \leq \| \Lambda^{1-\alpha/2}  \bfUe_0  \|^2 + \| \Lambda^{1-\alpha/2}   \sigma \|^2 t.
   \label{eq:k:b:i:l:bnd}
\end{align}
Hence, by applying the Krylov-Bogoliubov averaging procedure we immediately infer, for all $\epsilon$ strictly positive, that there existence of an invariant
$\mu^\epsilon$ for the Markov semigroup $P^\epsilon$ associated with \eqref{eq:e:v:vis:1}.
Noting that the bound \eqref{eq:exp:mar:3} also holds
for $\qe$ with all of the constants independent of $\epsilon > 0$ and we infer
\begin{align}
  \sup_{\epsilon > 0} \int \exp(\eta \| \Lambda^{1 - \alpha/2} \bfU \|^2) d \mu^\epsilon(\bfU)  \leq C < \infty.
  \label{eq:uni:2d:ev}
\end{align}
We have thus established the condition \eqref{eq:uni:mom:bnd:cond} for the collection of invariant measure for $P^\epsilon$.   The existence now
follows once we establish \eqref{eq:weak:u:conv} in our setting.

For this purpose fix any initial condition $\bfU_0 \in H^{1 - \alpha/2}$.
Observe that $\bfV^\epsilon = \bfU(t, \bfU_0) - \bfUe(t,\bfU_0)$ satisfies
\begin{align}
  \pd{t} \bfV^\epsilon +  \gamma \bfV^\epsilon
  +  \bfV_\alpha^\epsilon \cdot \nabla \bfU_\alpha + \bfU_\alpha^\epsilon \cdot \nabla \bfV_\alpha^\epsilon + \nabla p + \epsilon \Delta \bfUe = 0, \quad \bfV^\epsilon(0) =0.
  \label{eq:d:diff:apx:ev}
\end{align}
Similarly to above in \eqref{eq:2d:bnd:1},
\begin{align*}
  \frac{1}{2} \frac{d}{dt} \|\Lambda^{-\alpha/2} \bfV^\epsilon\|^2 +  \gamma \|\Lambda^{-\alpha/2}  \bfV^\epsilon\|^2 \leq C \| \Lambda^{-\alpha/2} \bfV^\epsilon \|^2  \|\Lambda^{-\alpha/2} \xi \|
  + \epsilon  \|\Delta \Lambda^{-\alpha/2} \bfUe\| \|\Lambda^{-\alpha/2} \bfV^\epsilon\|
\end{align*}
and hence
\begin{align*}
  \frac{1}{2} \frac{d}{dt} \|\Lambda^{-\alpha/2} \bfV^\epsilon\| \leq C \| \Lambda^{-\alpha/2} \bfV^\epsilon \|  \|\Lambda^{-\alpha/2} \xi \|
  + \epsilon  \|\Delta \Lambda^{-\alpha/2} \bfUe\|
\end{align*}
which implies
\begin{align}
   \|\Lambda^{-\alpha/2} \bfV^\epsilon\| \leq  \sqrt{\epsilon} \exp\left( Ct + \frac{\eta \gamma}{2} \int_0^t  \|\Lambda^{-\alpha/2} \xi \|^2 ds \right)
    \int_0^t  \sqrt{\epsilon} \|\nabla \Lambda^{-\alpha/2} \qe\| ds.
\end{align}
Taking expected values we find
\begin{align}
  \E \|\Lambda^{-\alpha/2} \bfV^\epsilon\| \leq  \sqrt{\epsilon} \sqrt{t} \left( \E \exp\left( Ct + \eta \gamma \int_0^t  \|\Lambda^{-\alpha/2} \xi \|^2 ds \right) \right)^{1/2}
   \left( \epsilon \E \int_0^t  \|\nabla \Lambda^{-\alpha/2} \qe\|^2 ds \right)^{1/2}.
   \label{eq:eps:ev:conv:f}
\end{align}
Combining this bound with \eqref{eq:exp:mar:4} (which holds for solution of \eqref{eq:e:v:vis:1} with constant independent of $\epsilon > 0$)
and \eqref{eq:k:b:i:l:bnd} we conclude that
\begin{align}
  \E \|\Lambda^{-\alpha/2} (\bfU(t, \bfU_0) - \bfUe(t,\bfU_0)) \| \leq  \sqrt{\epsilon} \exp(C(t +  \| \Lambda^{1 - \alpha/2} \bfU_0\|^2))
\end{align}
for a constant $C$ independent of $t, \epsilon$ and $\bfU_0$.  The condition \ref{eq:weak:u:conv} now follows and in conclusion we have that
\begin{proposition}\label{prop:ex:IM}
  Assume that $\alpha \geq 2/3$ and consider \eqref{eq:e:v:1} in the case $n = 2$.   Then for any $\gamma > 0$ there exists at least one invariant measure
  $\mu$ of \eqref{eq:e:v:1}  such that
  \begin{align}
    \int \exp(\eta \| \Lambda^{1 - \alpha/2} \bfU \|^2) d \mu(\bfU)  < \infty.
  \end{align}
\end{proposition}
\subsubsection{Uniqueness of the Invariant Measure (n =2)}
In order to establish the uniqueness of the invariant measure identified in \eqref{prop:ex:IM} fix any $\bfU_0, \tbfU_0$.  We take $\bfU = \bfU(t, \bfU_0)$ as the associated
solution of \eqref{eq:e:v:1} and consider $\tbfU$ solving
\begin{align}
  d \tbfU + (\gamma \tbfU + \tbfU_\alpha \cdot \nabla \tbfU_\alpha + \nabla p) dt = \lambda P_N (\bfU - \tbfU) dt \indFn{t \leq \tau_R} dt +  \sum_{k =1}^N \sigma_k dW^k,
  \quad \tbfU(0) = \tbfU_0,
  \label{eq:e:v:t:1}
\end{align}
where $\tau_R$ is the stopping time
\begin{align}
   \tau_R  := \inf_{t \geq 0} \left\{  \int_0^t \lambda^2 \|P_N (\bfU - \tbfU)\|^2 dt  > R  \right\}.
\end{align}
Here $\lambda, R$ are parameters to be determined presently.  Let $\bfV = \bfU - \tbfU$ and observe that
\begin{align}
  \pd{t} \bfV +  \gamma \bfV + \bfU_\alpha \cdot \nabla \bfV_\alpha + \bfV_\alpha \cdot \nabla \bfU_\alpha +  \nabla p =  - \lambda P_N \bfV \indFn{t \leq \tau_R},
\end{align}
so that on the interval $[0, \tau_R]$
\begin{align}
  \frac{1}{2} \frac{d}{dt} \|\Lambda^{-\alpha/2} \bfV\|^2 +  \gamma \|\Lambda^{-\alpha/2}  \bfV\|^2 + \lambda \|P_N \Lambda^{-\alpha/2}  \bfV\|^2
   = - \int  \bfV_\alpha \cdot \nabla \bfU_\alpha \cdot   \bfV_\alpha dx.
  \label{eq:ap:4}
\end{align}
We suppose that $\alpha > 2/3$ so that for some $\delta = \delta(\alpha) > 0$ we have that $H^{\alpha/2 - \delta} \subset L^3$.  As such, cf. \eqref{eq:2d:bnd:1}, we have
from the inverse poincare inequality that
\begin{align}
 &\left| \int  \bfV_\alpha \cdot \nabla \bfU_\alpha \cdot   \bfV_\alpha dx  \right| \notag\\
 &\qquad \qquad
 	\leq C (\|  P_N \Lambda^{-\alpha} \bfV \|_{H^{\alpha/2 }}^2 + \|  Q_N \Lambda^{-\alpha} \bfV \|_{H^{\alpha/2 - \delta}}^2  ) \|\Lambda^{-\alpha} \xi \|_{H^{\alpha/2}} \notag\\
 &\qquad \qquad \leq \lambda\|  P_N \Lambda^{-\alpha} \bfV \|_{H^{\alpha/2 }}^2
 + \frac{C}{\lambda}\| \Lambda^{-\alpha} \bfV \|_{H^{\alpha/2 }}^2 \|\Lambda^{-\alpha} \xi \|_{H^{\alpha/2}}^2
 + \frac{C}{N^\delta}  \| \Lambda^{-\alpha} \bfV \|_{H^{\alpha/2 }}^2   \|\Lambda^{-\alpha} \xi \|_{H^{\alpha/2}}.
 \label{eq:2d:bnd:2}
 \end{align}
Combining this bound with \eqref{eq:ap:4} and rearranging we find that $[0, \tau_R]$
\begin{align}
  \frac{1}{2} \frac{d}{dt} \|\Lambda^{-\alpha/2} \bfV\|^2 +  \left( \gamma - C(\lambda^{-1} + N^{-\delta}) (1 + \|\Lambda^{-\alpha} q \|_{H^{\alpha/2}}^2)\right)
  \|\Lambda^{-\alpha/2}  \bfV\|^2 \leq 0
\end{align}
We emphasize that $C$ depends only on quantities coming from Sobolev embedding and that $\delta$ only depends on $\alpha$.
Both quantities are independent of our choice of $R >0$.
Thus, by choosing $\lambda$ and $N$ sufficiently large (depending again only on $\alpha$, $\gamma$, $\|\Lambda^{-\alpha/2} \rho\|^2$ and universal quantities), we obtain
the bound
\begin{align}
	\|\Lambda^{-\alpha/2} \bfV(t \wedge \tau_R)\|^2 \leq
	\exp \left( - \frac{\gamma}{2} t\wedge \tau_R
	+ \frac{\gamma \min \{\gamma, 1\} }{4 \max\{\|\Lambda^{-\alpha/2} \rho\|^2, 1 \} }  \int_0^{t\wedge \tau_R}\! \! \! \! \! \!\|\Lambda^{-\alpha} \xi \|_{H^{\alpha/2}}^2 ds \right)
	\| \Lambda^{-\alpha/2}( \bfU_0 -\tbfU_0)\|^2
\end{align}
This implies that on the set
\begin{align}
 E_K := \left\{ \sup_{t \geq 0} \left(\| \Lambda^{-\alpha/2} \xi(t)\|^2+  \gamma \int_0^t \| \Lambda^{-\alpha/2} \xi \|^2 ds -
   (\|\Lambda^{-\alpha/2} \rho\|^2 t  + \|\Lambda^{-\alpha/2} \xi_0\|^2) \right) \leq K  \right\}
\end{align}
we have
\begin{align}
	\|\Lambda^{-\alpha/2} \bfV(t \wedge \tau_R)\|^2 \leq
	\exp \left( - \frac{\gamma}{4} t\wedge \tau_R
	+ K + \|\Lambda^{-\alpha/2} \xi_0\|^2 \right)
	\| \Lambda^{-\alpha/2}( \bfU_0 -\tbfU_0)\|^2.
\end{align}
By now choosing $K$ large enough that $\Prb(E_K) > 1/2$ and then taking $R$ sufficiently large we now obtain
\begin{proposition}
 Consider \eqref{eq:e:v:1} in the case $n = 2$.   Then for any $\gamma > 0$ and any $\alpha > 2/3$ there exists
 an $N = N(\alpha, \gamma, \|\Lambda^{-\alpha/2} \rho\|^2)$ such that if $H_N \subset \mbox{Range}(\sigma)$
 then \eqref{eq:e:v:1} has at most one invariant measure.
\end{proposition}

\begin{remark}[The Three Dimensional Case]\label{rmk:3d:case}
As already mentioned the approach taken here also yields the existence and uniqueness
of invariant measures for \eqref{eq:e:v:1}-\eqref{eq:e:v:2} in dimensional three whenever $\alpha \geq 2$.
The following modifications of the proof are required primarily as a consequence of the fact that we are not able to
make use of the vorticity formulation in 3D as above in \eqref{eq:ap:2}.    Firstly we
note that we consider solutions $u \in L^2(\Omega; L^\infty([0,\infty); H^{-\alpha/2})$.  Taking $\bfV$ to be the difference of
two solutions, uniqueness
and continuous dependence on data in $H^{-\alpha/2}$ follows from the estimate
\begin{align}
\left| \int  \bfV_\alpha \cdot \nabla \bfU_\alpha \cdot   \bfV_\alpha dx  \right| \leq \| \nabla \bfU_\alpha \| \|  \bfV_\alpha \|_{L^4}^2
\leq \| \Lambda^{-\alpha/2} \bfU \| \|\Lambda^{3/4 -\alpha} \bfV \|^2 \leq \| \Lambda^{-\alpha/2} \bfU \| \|\Lambda^{ -\alpha/2} \bfV \|^2
\label{eq:ev:un:3d}
\end{align}
which we may combine with \eqref{eq:ap:1} to close \eqref{eq:ap:3}.  The estimates leading to the existence
of an invariant measure are also a little different.  Here we take $V = H^{-\alpha/2}$ and $H = H^{-\alpha}$.
The bounds \eqref{eq:k:b:i:l:bnd} and \eqref{eq:uni:2d:ev} are replaced with
\begin{align*}
   \sup_{\epsilon > 0, t\geq 1}  \epsilon \E \frac{1}{t}\int_0^t  \|  \Lambda^{1-\alpha/2}  \bfUe \|^2 ds
   + \sup_{\epsilon > 0} \int \exp(\eta \| \Lambda^{- \alpha/2} \bfU \|^2) d \mu^\epsilon(\bfU)  \leq C < \infty.
\end{align*}
and the convergence $\bfU^\epsilon \to \bfU$ is now carried out in the $H^{-\alpha}$ topology.  For the convergence,
taking $\bfV^\epsilon = \bfU - \bfU^\epsilon$,
\eqref{eq:d:diff:apx:ev} leads to
\begin{align*}
  \frac{1}{2} \frac{d}{dt} \|\Lambda^{-\alpha} \bfV^\epsilon\|^2 +  \gamma \|\Lambda^{-\alpha}  \bfV^\epsilon\|^2 &\leq
  \left|\int (\bfV^\epsilon_\alpha \cdot \nabla \bfU_\alpha + \bfU_\alpha^\epsilon \cdot \nabla \bfV^\epsilon_\alpha) \Lambda^{-2\alpha} \bfV^\epsilon  dx\right |
  + \epsilon  \|\Delta \Lambda^{-\alpha} \bfUe\| \|\Lambda^{-\alpha} \bfV^\epsilon\| \\
  &\leq C \| \Lambda^{-\alpha} \bfV^\epsilon \|^2 (\| \Lambda^{-\alpha/2} \bfU \| + \| \Lambda^{-\alpha/2} \bfU^\epsilon \|)
  + \epsilon  \| \Lambda^{1-\alpha/2} \bfUe\| \|\Lambda^{-\alpha} \bfV^\epsilon\|.
\end{align*}
so that the convergence required by the abstract condition \eqref{eq:weak:u:conv} now follows in a similar fashion to
\eqref{eq:eps:ev:conv:f} above.
Finally regarding the uniqueness, the strategy is essentially the same once we notice that \eqref{eq:ev:un:3d} provides
the sub-criticality necessary to replace the estimate \eqref{eq:2d:bnd:2}.
\end{remark}
\subsection{A Damped Nonlinear Wave Equation}

Our final example is the damped Sine-Gordon equation which we write formally as
\begin{align}
  \partial_{tt} u + \alpha \partial_t u - \Delta u + \beta \sin(u) = \sum_{k =1}^d \sigma_k \dot{W}^k.
    \label{eq:sign:gordon:formal}
\end{align}
Here the unknown $u$ evolves on a bounded domain $\DD \subset \RR^n$ with smooth
boundary and satisfies the Dirichlet boundary condition $u_{\partial \DD} \equiv 0$.
The parameter $\alpha$ is strictly positive and $\beta$ is a given real number.
The functions $\sigma_k$ on $\DD$ will be specified below, and $\dot{W}^k$ represent a sequence of independent white noise processes.
This is written more rigorously as the system of stochastic partial differential equations
\begin{align}
  &d v + (\alpha v - \Delta u + \beta \sin(u))dt = \sum_{k =1}^d \sigma_k dW^k, \quad \frac{du}{dt}  = v,
  \label{eq:sign:gordon:rig}
\end{align}
which we supplement with the initial condition $u(0) = u_0, v(0) = v_0$.

The deterministic Sine-Gordon equation appears in the description of continuous Josephson junctions \cite{levi1978},
and has been studied extensively in a variety of contexts \cite{bishop1983,bishop1986,caraballo2004,dickey1976,fan2004,ghidaglia1987,kovavcivc1992,wang1997}.
For example, analysis of the existence and finite dimensionality of the attractor for the deterministic
counterpart of \eqref{eq:sign:gordon:formal} can be found in \cite{Temam1997}.

\subsubsection{Mathematical Preliminaries}

For any given $(u_0, v_0) \in X := H^1_0(\DD) \times L^2(\DD)$ there exists
a unique $U = (u,v) \in L^2(\Omega; C([0,\infty), X)$
which is a (weak) solution of \eqref{eq:sign:gordon:rig}.   These solutions $U(t) = U(t, U_0)$
depend continuously on $U_0 = (u_0, v_0) \in X$ and hence
$P_t \phi(U_0) := \E \phi(U(t,U_0))$ is a Feller Markov semigroup acting on $C_b(X)$.
Moreover when $(u_0, v_0) \in Y := (H^2(\DD) \cap H^1_0(\DD)) \times H^1_0(\DD)$
the corresponding solution satisfies $U  \in L^2(\Omega; C([0,\infty), Y)$.
In what follows we will
maintain the standing convention that $| \cdot | = \| \cdot \|_{L^2}$ and $\| \cdot \| = \| \cdot \|_{H^1}$ with
all other norms given explicitly.

The existence of solutions may be established via standard compactness
methods starting from a Galerkin truncation of
\eqref{eq:sign:gordon:rig} and making use of the following a priori estimates.  Take $\vshft = v +  \epsilon u$ with $\epsilon > 0$ to be
specified presently. Evidently
\begin{align}
  d \vshft + (\alpha - \epsilon) \vshft dt =  \left(\epsilon(\alpha - \epsilon) u + \Delta u - \beta \sin(u)\right)dt + \sum_{k=1}^d \sigma_k dW^k.
  \label{eq:vshft:eqn}
\end{align}
From the It\={o} lemma we infer
\begin{align*}
  d |\vshft|^2 + 2(\alpha - \epsilon) |\vshft|^2 dt =  \left( 2\epsilon(\alpha - \epsilon)  \langle u, \vshft \rangle + 2\langle \Delta u, \vshft \rangle - 2\beta \langle \sin(u), \vshft \rangle
  + |\sigma|^2 \right)dt + 2\langle \sigma, \vshft \rangle dW.
\end{align*}
Now since
\begin{align*}
  2\langle \Delta u, \vshft \rangle =  - \frac{d}{dt} \| u \|^2 - 2\epsilon \|u\|^2,
\end{align*}
we infer that when $\epsilon \leq \alpha/2$
\begin{align*}
  d (|\vshft|^2 + \| u \|^2) + (\alpha |\vshft|^2 + 2\epsilon \|u\|^2 ) dt  \leq  \left(\frac{\epsilon \alpha}{\sqrt{\lambda}}  \|u\| |r| + 2 |\beta| | \DD|^{1/2} |r| + |\sigma|^2\right)  dt
  	+ 2 \langle \sigma, \vshft \rangle dW,
\end{align*}
where $\lambda = \lambda(\DD)$ is the Poincar\'e constant.
By now choosing
\begin{align}
  \epsilon := \min \left\{  \frac{\lambda}{\alpha}, \frac{\alpha}{2}, \sqrt{\frac{\lambda}{2 }}\right\},
  \label{eq:eps:decision}
\end{align}
we have that
\begin{align}
  d (|\vshft|^2 + \| u \|^2) + \epsilon \left( |\vshft|^2 +  \|u\|^2 \right) dt  \leq   \Big( \frac{4 |\beta|^2 | \DD|}{\alpha}  + |\sigma|^2 \Big) dt + 2 \langle \sigma, \vshft \rangle dW.
  \label{eq:diff:eq:wave:low}
\end{align}
and that
\begin{align}
  \frac{1}{2}(|v|^2 + \|u\|^2) \leq |r|^2 + \|u\|^2\leq 2(|v|^2 + \|u\|^2).
  \label{eq:r:v:eqiv}
\end{align}
Combining the previous two inequalities and using the exponential Martingale bound, \eqref{eq:exp:mart}, we conclude
\begin{align}
  \Prb\biggl( \sup_{t \geq 0}
   \Big[ \frac{1}{2} |v(t)|^2 + \| u(t)   \|^2 + \frac{\epsilon}{4}  &\int_0^t (|v(s) |^2  +  \|u(s)\|^2) \, ds \notag\\
  &- \big(\frac{4 |\beta|^2 |\DD|}{\alpha}  + |\sigma|^2 \big)  t - 2(|v_0|^2 + \|u_0\|^2)   \Big]
  \geq  K  \biggr)  \leq e^{-\gamma K}, \label{eq:small:prob:wave:weak:norm}
\end{align}
for every $K > 0$ where $\gamma = \gamma(|\sigma|, \alpha) > 0$ is independent of $K$ and of the solution $U = (u,v)$.

In order to prove the existence of an invariant measure for $\{P_t\}_{t \geq 0}$ we next establish suitable bounds
for $U = (u,v)$ in $Y = (H^2(\DD) \cap H_0^1(\DD)) \times H_0^1(\DD)$.  Denote $- \Delta$ with Dirchlet boundary conditions as $A$.
Applying $A^{1/2}$ to \eqref{eq:vshft:eqn} and then invoking the It\=o lemma we obtain
\begin{align*}
  d \|\vshft\|^2 + 2(\alpha - \epsilon) \|\vshft\|^2 dt =&
  \left( 2\epsilon (\alpha - \epsilon)\langle A^{1/2} u, A^{1/2}\vshft \rangle - 2\langle \Delta u, \Delta \vshft \rangle - 2\beta \langle A^{1/2} \sin(u), A^{1/2}\vshft \rangle
  + \|\sigma\|^2 \right)dt \\
  &+ 2\langle A^{1/2} \sigma, A^{1/2} \vshft \rangle dW,
\end{align*}
and hence estimating as above and imposing the same condition on $\epsilon$ we find
\begin{align}
  d (\|\vshft \|^2 + | A u |^2) + \epsilon (\|\vshft\|^2 +  |Au|^2) dt  \leq  \Big(\|\sigma\|^2 +  \frac{4 |\beta|^2}{\alpha} \|u\|^2
  \Big) dt + 2  \langle A^{1/2} \sigma, A^{1/2} \vshft \rangle dW.  \label{eq:diff:eq:wave:high}
\end{align}
Combining \eqref{eq:diff:eq:wave:high} and  \eqref{eq:diff:eq:wave:low} and noting that, similarly to \eqref{eq:r:v:eqiv},
$\|r\|^2 + \|u\|^2_{H^2}\leq 2(\|v\|^2 + \|u\|^2_{H^2})$ we now infer
\begin{align*}
 \int_0^T \E (\|v(t)\|^2 +  \|u(t)\|^2_{H^2})dt
\leq& C \Big(\int_0^T(\E\|u(t)\|^2 + 1)dt \Big)  \leq CT,
\end{align*}
for any $T > 0$ when $u_0 = v_0 \equiv 0$.  Here the constant $C = C(\sigma, \beta, \alpha, \DD)$ but is independent of $T$.
The existence of an ergodic invariant measure $\mu \in Pr(X)$
for \eqref{eq:sign:gordon:rig} now follows from the
Krylov-Bogolyubov theorem.

\subsubsection{Asymptotic Coupling Arguments}

To establish the uniqueness of invariant measures for \eqref{eq:sign:gordon:rig}
we fix arbitrary $U_0, \tilde{U}_0 \in X = H^1_0(\DD) \times L^2(\DD)$.  Take
$U = (u, v)$ to be the solution of \eqref{eq:sign:gordon:rig} corresponding to $U_0$
and let $\tilde{U} = (\tilde{u}, \tilde{v})$ be the solution of
\begin{align}
  d \tilde{v} + (\alpha \tilde{v} - \Delta \tilde{u} + \beta \sin(\tilde{u}) -\beta\indFn{\tau_{K} > t} P_N (\sin(u)-\sin(\tilde{u})) )dt = \sum_k \sigma_k dW^k, \quad \frac{d}{dt} \tilde{u} = \tilde{v}   \label{eq:sign:gordon:rig:shft}
\end{align}
where $\tilde{u}(0) = \tilde{u}_0, \tilde{v}(0) = \tilde{v}_0$, and
\begin{align*}
  \tau_K := \inf_{t \geq 0}  \left\{   \int_0^t |u-\tilde{u}|^2 ds \geq K \right\}.
\end{align*}
In the framework of Section~\ref{sec:recipeC}, we have taken
$G(u,\tilde{u})=P_N (\sin(u)-\sin(\tilde{u})) )$ rather than $\lambda
P_N(u-\tilde{u})$ as in the preceding sections.
It follows that $h(t) = \indFn{\tau_{K} > t}\sigma^{-1}\beta P_N (\sin(u)-\sin(\tilde{u})) $ is a continuous adapted process in $\RR^N$ which
satisfies the Novikov condition \eqref{eq:st:nov:cond}.
Taking $w = u - \tilde{u}$ and subtracting \eqref{eq:sign:gordon:rig:shft} from \eqref{eq:sign:gordon:rig} we obtain
\begin{align}
  \partial_{tt} w + \alpha \partial_{t} w - \Delta w  = \beta (\sin(\tilde{u}) - \sin(u)) -  \indFn{\tau_{K} > t}\beta P_N (\sin(\tilde{u})-\sin(u)).
\end{align}
Modifying slightly the method of previous examples, uniqueness of the invariant measure will follow from showing that for $N,K>0$ sufficiently large,  $\tau_{K}=\infty$ almost surely, and moreover
\begin{align}
   |\partial_{t} w(t)|^2 + \|w(t)\|^2 \to 0 \textrm{ as } t \to \infty.
   \label{eq:good:con}
\end{align}

To this end, we again pursue the strategy leading to \eqref{eq:diff:eq:wave:low}, \eqref{eq:diff:eq:wave:high}
and introduce $\yshft =  \partial_t w + \epsilon w$ with $\epsilon$ as in \eqref{eq:eps:decision}.  Similarly
to \eqref{eq:vshft:eqn} above $y$ satisfies
\begin{align}
  \partial_t \yshft + (\alpha - \epsilon) \yshft - \Delta w  =  \epsilon(\alpha - \epsilon)  w  + \beta (\sin(\tilde{u}) - \sin(u)).
  \label{eq:yshft:eqn}
\end{align}
This equation can be projected to low and high frequencies, giving
\begin{align}
\partial_t P_N\yshft + (\alpha - \epsilon) P_N\yshft - \Delta P_N w   &=  \epsilon(\alpha - \epsilon) P_N w +
   \indFn{\tau_{K} \leq t}\beta P_N (\sin(\tilde{u})-\sin(u)).
  \label{eq:yshft:pn}\\
\partial_t Q_N\yshft + (\alpha - \epsilon) Q_N\yshft - \Delta Q_N w   &=  \epsilon(\alpha - \epsilon) Q_N w  + \beta Q_N(\sin(\tilde{u}) - \sin(u)).
  \label{eq:yshft:qn}
\end{align}
Multiplying these expressions by $\yshft$, and integrating over $\DD$, when $t<\tau_K$ this gives
\begin{align}
  \frac{d}{dt} (|P_N\yshft|^2 + \|P_N w\|^2) +  \epsilon (|P_N\yshft|^2 + \|P_N w\|^2)
  &\leq   0,
  \label{eq:FP:wave:bnd:p}
  \\
    \frac{d}{dt} (|Q_N\yshft|^2 + \|Q_N w\|^2) +  \epsilon (|Q_N\yshft|^2 + \|Q_N w\|^2)
  &\leq
   \beta \langle Q_N(\sin(\tilde{u}) - \sin(u)),Q_N \yshft\rangle.
  \label{eq:FP:wave:bnd:q}
\end{align}
By Gr\"{o}nwall's inequality
\begin{align}
(|P_N\yshft|^2 + \|P_N w\|^2)(t\wedge \tau_K) \leq e^{-\epsilon t\wedge \tau_K}(|P_N\yshft_0|^2 + \|P_N w_0\|^2),
\label{eq:l:m:sg:bnd}
\end{align}
and using the inverse Poincar\'{e} inequality, taking $N = N(\beta,\epsilon)$ sufficiently large, we find
\begin{align*}
  \frac{d}{dt} (|Q_N\yshft|^2 + \|Q_N w\|^2) +  \epsilon (|Q_N\yshft|^2 + \|Q_N w\|^2)
  &\leq
   |\beta| |w||Q_N \yshft| \leq |\beta| |P_N w||Q_N \yshft| + |\beta| |Q_N w||Q_N \yshft| \\
  &\leq \frac{\epsilon}{4}|Q_N\yshft|^2  + C_{\epsilon}|\beta|^{2} |P_N w|^{2} + \frac{|\beta|}{\lambda_N}\|Q_N w\||Q_N \yshft| \\
  &\leq \frac{\epsilon}{2}(|Q_N\yshft|^2 + \|Q_N w\|^2) + C_{\epsilon}|\beta|^{2} |P_N w|^{2}.
\end{align*}
Applying Gr\"{o}nwall once more and then making use of  \eqref{eq:l:m:sg:bnd} we find that for for $t<\tau_K$,
\begin{align}
(|Q_N\yshft|^2 + \|Q_N w\|^2)(t) &\leq e^{-\frac{\epsilon}{2}t} (|Q_N\yshft_0|^2 + \|Q_N w_0\|^2)
+ C_{\epsilon}|\beta|^{2}\int_{0}^{t}e^{-\frac{\epsilon}{2}(t-s)}|P_N w(s)|^{2}ds \\
&\leq  e^{-\frac{\epsilon}{2}t} (|Q_N\yshft_0|^2 + \|Q_N w_0\|^2)
+ \tilde{C}_{\epsilon}|\beta|^{2}e^{-\frac{\epsilon}{2}t}(|P_N\yshft_0|^2 + \|P_N w_0\|^2).
\end{align}
Combining the estimates on the high and low modes,
\begin{multline}
(|\yshft|^2 + \|w\|^2)(t\wedge\tau_K)  \leq e^{-\frac{\epsilon}{2}t\wedge\tau_K}\left((e^{-\frac{\epsilon}{2} t\wedge\tau_K}+ \tilde{C}_{\epsilon}|\beta|^{2})(|P_N\yshft_0|^2 + \|P_N w_0\|^2) +  |Q_N\yshft_0|^2 + \|Q_N w_0\|^2\right),
\end{multline}
and we conclude that $\tau_K=\infty$ almost surely for $K$ sufficiently large.
Moreover, due to \eqref{eq:r:v:eqiv} the convergence $|\pd{t}w|^2 + \|w\|{2} \leq 2(|\yshft|^2 + \|w\|^2) \rightarrow 0$ is obtained, almost surely.

In summary we have proven the following result
\begin{proposition}
For every $\alpha > 0$, $\beta \in \RR$ and $N \geq 0$ \eqref{eq:sign:gordon:rig} possesses an ergodic invariant measure $\mu$.
Moreover for each $\alpha > 0$ and $\beta \in \RR$ there exists an $N = N(\alpha, |\beta|)$ such that
if $\mathrm{Range}(\sigma) \supset P_N L^2(\DD)$, then $\mu$ is unique.
\end{proposition}

\section*{Acknowledgments}
This work was partially supported by the National Science Foundation under the grant (NEGH)
NSF-DMS-1313272.   JCM was partially supported
by the Simmons Foundation. We would like to thank Peter Constantin, Michele Coti-Zelati, Juraj Foldes and Vlad Vicol
for helpful feedback.  We would also like to express our appreciation to the Mathematical Sciences Research Institute
(MSRI) as well as the Duke and Virginia Tech Math Departments where
the majority of this work was carried out.

\appendix
\section{Existence of Invariant Measures by a Limiting Procedure}
\label{sec:ex:IM}
We now present some abstract results which are used above to infer the existence of an invariant
measure via an approximation procedure relying on invariant measures for a collection
of regularized systems.

Let $(\spH, \| \cdot \|_\spH)$, $(\spV, \| \cdot \|_\spV)$ be two separable Banach spaces.
The associated Borel $\sigma$-algebras are denoted as $\mathcal{B}(H)$ and $\mathcal{B}(V)$ respectively.
We suppose that $V$ is continuously and compactly embedded in $\spH$.
Moreover we assume  that there exists continuous functions $\rho_n:
H \to V$ for $n \geq 1$ such that
\begin{align*}
  \lim_{n \to \infty} \|\rho_n(u)\|_V =
  \begin{cases}
     \| u\|_V & \textrm{ for } u \in V\\
     \infty & \textrm{ for } u \in H \setminus V.
  \end{cases}
\end{align*}
Notice that, under these circumstances,
$\mathcal{B}(\spV) \subset \mathcal{B}(\spH)$ and moreover that
$A \cap \spV \in \mathcal{B}(\spV)$ for any $A \in \mathcal{B}(\spH)$.
We can therefore extend any Borel measure $\mu$ on $V$ to a measure
$\mu_E$ on $H$ by setting $\mu_E(A) = \mu(A \cap V)$ and hence we
identify $Pr(V) \subset Pr(H)$.  This natural extension will be made
without further comment in what follows.

By appropriately restricting the domain of elements $\phi \in C_b(H)$
to $V$ we have that $C_b(H) \subset C_b(V)$.  Similarly
$\mathrm{Lip}(H) \subset \mathrm{Lip}(V)$, etc.  Furthermore, under
the given conditions on $H$ and $V$, $C_b(H) \cap \mathrm{Lip}(H)$
determines measures in $Pr(V)$ namely if
$\int_V \phi \, d\mu = \int_V \phi \, d \nu$ for all
$\phi \in C_b(H) \cap \mathrm{Lip}(H)$ then $\mu = \nu$.

On $\spV$ we consider a Markov transition
kernel $P$, which is assumed to be Feller in $H$, that is to say
$P$ maps $C_b(\spH)$ to itself.
We also suppose that $\{P^\epsilon\}_{\epsilon > 0}$ is a sequence of
Markov transition kernels (again defined on $V$) such that, for any $\phi \in C_b(H) \cap \mathrm{Lip}(H)$,
and $R>0$,
\begin{align}
  \lim_{\epsilon \rightarrow 0} \sup_{u \in B_{R}(\spV)} | P^\epsilon\phi(u) - P\phi(u)| = 0,
  \label{eq:weak:u:conv}
\end{align}
where $B_R(V)$ is the ball of radius $R$ in $V$.

  \begin{lemma}\label{randomDataOK}
  In the above setting, let $\{\mu^\epsilon\}_{\epsilon > 0}$ be a sequence of probability  measures on $\spV$.
  Assume that there is an increasing continuous function
$\psi:[0,\infty) \rightarrow [0,\infty)$ with $\psi(r) \rightarrow
\infty$ as $r\rightarrow \infty$ and a finite constant $C_0 >0$ so that
\begin{align}
  \sup_{\epsilon > 0} \int \psi(\|u\|_{\spV}) d \mu^\epsilon \leq C_0.
  \label{eq:uni:mom:bnd:cond}
\end{align}
Then there exists a probability measure $\mu$, supported on $V$, with
$\int \psi(\|u\|_{\spV}) d \mu(u) \leq C_0$ such that (up to a subsequence)
$\mu^\epsilon P^\epsilon$ converges weakly  in $\spH$
    to  $\mu P$ that is, for all $\phi \in C_b(\spH)$,
    \begin{align}
     \lim_{\epsilon\rightarrow0}  \mu^\epsilon P^\epsilon\phi = \mu P\phi
     \label{eq:weak:muP:conv}
    \end{align}
  \end{lemma}
  \begin{proof}[Proof of Lemma~\ref{randomDataOK}]
    From our assumption we know that
    \begin{align}
    \mu^\epsilon (\psi(\|u\|_{\spV}) \geq R) \leq C_0/\psi(R)
    \label{eq:cheb:u:bnd}
    \end{align}
    for all $\epsilon> 0$. We infer that
  the family of measures $\{\mu^\epsilon\}_{\epsilon >0}$ is tight
  on $\spH$ and thus that there
  exists a measure $\mu$ on $\spH$ such that $\mu^{\epsilon_n}$ converges
  weakly in H to $\mu$ for some decreasing subsequence
  $\epsilon_n \to 0$.   For $k, m \geq 1$ define $f_{k,m} \in C_b(H)$
  as $f_{k,m}(u) := \psi(\|\rho_m (u)\|_{\spV}) \wedge k$.   Weak convergence in $H$ implies
  that $\int f_{k,m} d \mu^{\epsilon_n}
  \rightarrow \int f_{k,m} d \mu \leq C_0$ as $n \to \infty$ for each fixed $k, m$.  Fatou's
  lemma then implies that
  \begin{align*}
  \int \psi(\|u\|_{\spV})  d\mu(u) \leq  \lim_{k, m \rightarrow \infty}  \int f_{k,m}(u)  d\mu(u) \leq C_0
  \end{align*}
  and in particular that $\mu(\spV)=1$.

   We now turn to demonstrate \eqref{eq:weak:muP:conv}.   Observe
   that, for any $\phi \in C_b(H)$ and any $\epsilon>0$,
\begin{align}
 \big|  \mu^{\epsilon}   P^{\epsilon}\phi - \mu  P\phi \big| & \leq
      \left| \mu^{\epsilon} P^{\epsilon}\phi   -  \mu^{\epsilon}  P\phi  \right| + \left| \mu^{\epsilon}  P\phi \  -  \mu P\phi \right|
      \label{eq:weak:p:bnd:1}
\end{align}
Taking $\epsilon=\epsilon_n$, the first term is bounded as
    \begin{align}
        \left| \mu^{\epsilon_n}  P^{\epsilon_n}\phi  -   \mu^{\epsilon_n} P\phi  \right|
        & \leq\sup_{u \in B_{S}(\spV)}|  P^{\epsilon_n} \phi(u)- P\phi(u) | +
                                                     2 \sup_{u} | \phi(u)| \, \mu^{\epsilon_n} ( B_{S}(\spV)^c)
                                                           \label{eq:weak:p:bnd:2}
    \end{align}
    for any $S > 0$.
    Combining \eqref{eq:weak:p:bnd:1}, \eqref{eq:weak:p:bnd:2} with \eqref{eq:weak:u:conv}, \eqref{eq:cheb:u:bnd}, using that $\mu^{\epsilon_n}$ converges
weakly in $\spH$ and that $ P^\epsilon\phi \in C_b(\spH)$ we infer \eqref{eq:weak:muP:conv}, completing the proof.
  \end{proof}

This produces the following corollary.
\begin{corollary}\label{cor:ex:c}
  In the above setting, if in addition we assume that, for every $\epsilon >0$,
  $\mu^\epsilon$ is
  an invariant measure for $P^\epsilon$ then the limiting measure $\mu$ is an invariant
  measure of $P$.
\end{corollary}
\begin{proof}
By the above result we may pick $\epsilon_n \to 0$ such that $\mu^{\epsilon_n}$ and  $\mu^{\epsilon_n}
  P^{\epsilon_n}$ converge weakly in $\spH$ to $\mu$ and $\mu P$
  respectively. However since $\mu^{\epsilon_n}
  P^{\epsilon_n}=\mu^{\epsilon_n}$ we also have that $\mu^{\epsilon_n}
  P^{\epsilon_n}$ converges weakly  in $\spH$ to $\mu$. Hence we
  conclude that $\mu P=\mu$ which is means the $\mu$ is an invariant
  measure for $P$.
\end{proof}

\addcontentsline{toc}{section}{References}
\begin{footnotesize}
\def\cprime{$'$}

\end{footnotesize}


\begin{thebibliography}{GHKVZ14}

\bibitem[AFS08]{CIME08}
S.~Albeverio, F.~Flandoli, and Y.G. Sinai.
\newblock {\em S{PDE} in hydrodynamic: recent progress and prospects}, volume
  1942 of {\em Lecture Notes in Mathematics}.
\newblock Springer-Verlag, Berlin, 2008.
\newblock Lectures given at the C.I.M.E. Summer School held in Cetraro, August
  29--September 3, 2005, Edited by Giuseppe Da Prato and Michael R{{\"o}}ckner.

\bibitem[BFLT83]{bishop1983}
Alan~R Bishop, Klaus Fesser, Peter~S Lomdahl, and Steven~E Trullinger.
\newblock Influence of solitons in the initial state on chaos in the driven
  damped sine-gordon system.
\newblock {\em Physica D: Nonlinear Phenomena}, 7(1):259--279, 1983.

\bibitem[BKL01]{BricmontKupiainenLefevere2001}
J.~Bricmont, A.~Kupiainen, and R.~Lefevere.
\newblock Ergodicity of the 2{D} {N}avier-{S}tokes equations with random
  forcing.
\newblock {\em Comm. Math. Phys.}, 224(1):65--81, 2001.
\newblock Dedicated to Joel L. Lebowitz.

\bibitem[BKL02]{BricmontKupiainenLefevere02}
J.~Bricmont, A.~Kupiainen, and R.~Lefevere.
\newblock Exponential mixing of the 2{D} stochastic {N}avier-{S}tokes dynamics.
\newblock {\em Comm. Math. Phys.}, 230(1):87--132, 2002.

\bibitem[BL86]{bishop1986}
AR~Bishop and PS~Lomdahl.
\newblock Nonlinear dynamics in driven, damped sine-gordon systems.
\newblock {\em Physica D: Nonlinear Phenomena}, 18(1):54--66, 1986.

\bibitem[BM05]{BakhtinMattingly2005}
Yuri Bakhtin and Jonathan~C. Mattingly.
\newblock Stationary solutions of stochastic differential equations with memory
  and stochastic partial differential equations.
\newblock {\em Commun. Contemp. Math.}, 7(5):553--582, 2005.

\bibitem[CF88]{ConstantinFoias1988}
P.~Constantin and C.~Foias.
\newblock {\em Navier-{S}tokes equations}.
\newblock Chicago Lectures in Mathematics. University of Chicago Press,
  Chicago, IL, 1988.

\bibitem[CGHV13]{ConstantinGlattHoltzVicol2013}
P.~Constantin, N.~Glatt-Holtz, and V.~Vicol.
\newblock Unique ergodicity for fractionally dissipated, stochastically forced
  {2D} {E}uler equations.
\newblock {\em Comm. Math. Phys.}, 2013.
\newblock (to appear).

\bibitem[CHOT05]{CheskidovHolmOlsonTiti2005}
A.~Cheskidov, D.~D. Holm, E.~Olson, and E.~S. Titi.
\newblock On a leray--$\alpha$ model of turbulence.
\newblock In {\em Proceedings of the Royal Society of London A: Mathematical,
  Physical and Engineering Sciences}, volume 461, pages 629--649. The Royal
  Society, 2005.

\bibitem[CKR04]{caraballo2004}
Tomas Caraballo, Peter~E Kloeden, and Jose Real.
\newblock Pullback and forward attractors for a damped wave equation with
  delays.
\newblock {\em Stochastics and Dynamics}, 4(03):405--423, 2004.

\bibitem[CLT06]{CaoLunasinTiti2006}
Y.~Cao, E.~M. Lunasin, and E.~S. Titi.
\newblock Global well-posedness of the three-dimensional viscous and inviscid
  simplified bardina turbulence models.
\newblock {\em arXiv preprint physics/0608096}, 2006.

\bibitem[CT07]{CaoTiti2007}
C.~Cao and E.~Titi.
\newblock Global well-posedness of the three-dimensional viscous primitive
  equations of large scale ocean and atmosphere dynamics.
\newblock {\em Ann. of Math. (2)}, 166(1):245--267, 2007.

\bibitem[DGHT11]{DebusscheGlattHoltzTemam1}
A.~Debussche, N.~Glatt-Holtz, and R.~Temam.
\newblock Local martingale and pathwise solutions for an abstract fluids model.
\newblock {\em Physica D}, 2011.
\newblock (to appear).

\bibitem[DGHTZ]{DebusscheGlattHoltzTemamZiane1}
A.~Debussche, N.~Glatt-Holtz, R.~Temam, and M.~Ziane.
\newblock Global existence and regularity for the 3d stochastic primitive
  equations of the ocean and atmosphere with multiplicative white noise.
\newblock {\em Nonlinearity}.
\newblock (to appear).

\bibitem[Dic76]{dickey1976}
RW~Dickey.
\newblock Stability theory for the damped sine-gordon equation.
\newblock {\em SIAM Journal on Applied Mathematics}, 30(2):248--262, 1976.

\bibitem[DMKB15]{DiMolfettaKrstlulovicBrachet2015}
G.~Di~Molfetta, G.~Krstlulovic, and M.~Brachet.
\newblock Self-truncation and scaling in euler-voigt-$\alpha$ and related fluid
  models.
\newblock {\em arXiv preprint arXiv:1502.05544}, 2015.

\bibitem[DO05]{DebusscheOdasso2008}
Arnaud Debussche and Cyril Odasso.
\newblock Ergodicity for a weakly damped stochastic non-linear {S}chr\"odinger
  equation.
\newblock {\em J. Evol. Equ.}, 5(3):317--356, 2005.

\bibitem[DPZ96]{ZabczykDaPrato1996}
G.~Da~Prato and J.~Zabczyk.
\newblock {\em Ergodicity for infinite-dimensional systems}, volume 229 of {\em
  London Mathematical Society Lecture Note Series}.
\newblock Cambridge University Press, Cambridge, 1996.

\bibitem[EMS01]{EMattinglySinai2001}
W.~E, J.C. Mattingly, and Y.~Sinai.
\newblock Gibbsian dynamics and ergodicity for the stochastically forced
  {N}avier-{S}tokes equation.
\newblock {\em Comm. Math. Phys.}, 224(1):83--106, 2001.
\newblock Dedicated to Joel L. Lebowitz.

\bibitem[EPT07]{EwaldPetcuTemam}
B.~Ewald, M.~Petcu, and R.~Temam.
\newblock Stochastic solutions of the two-dimensional primitive equations of
  the ocean and atmosphere with an additive noise.
\newblock {\em Anal. Appl. (Singap.)}, 5(2):183--198, 2007.

\bibitem[Fan04]{fan2004}
Xiaoming Fan.
\newblock Random attractor for a damped sine-gordon equation with white noise.
\newblock {\em Pacific J. Math}, 216(1):63--76, 2004.

\bibitem[FGHRT13]{FoldesGlattHoltzRichardsThomann2013}
J.~Foldes, N.~Glatt-Holtz, G.~Richards, and E.~Thomann.
\newblock Ergodic and mixing properties of the {B}oussinesq equations with a
  degenerate random forcing.
\newblock {\em arXiv preprint arXiv:1311.3620}, 11 2013.

\bibitem[FHT02]{FoiasHolmTiti2002}
C.~Foias, D.~D. Holm, and E.~S. Titi.
\newblock The three dimensional viscous camassa--holm equations, and their
  relation to the navier--stokes equations and turbulence theory.
\newblock {\em Journal of Dynamics and Differential Equations}, 14(1):1--35,
  2002.

\bibitem[FKS87]{sinai87}
S.~W. Fomin, I.~P. Kornfel{\cprime}d, and J.~G. Sinaj.
\newblock {\em Teoria ergodyczna}.
\newblock Pa\'nstwowe Wydawnictwo Naukowe (PWN), Warsaw, 1987.
\newblock Translated from the Russian by Jacek Jakubowski.

\bibitem[FP67]{FoaisProdi67}
C.~Foias and G.~Prodi.
\newblock Sur le comportement global des solutions non-stationnaires des
  {\'e}quations de {N}avier-{S}tokes en dimension {$2$}.
\newblock {\em Rend. Sem. Mat. Univ. Padova}, 39:1--34, 1967.

\bibitem[GHKVZ14]{GlattHoltzKukavicaVicolZiane2014}
N.~Glatt-Holtz, I.~Kukavica, V.~Vicol, and M.~Ziane.
\newblock Existence and regularity of invariant measures for the three
  dimensional stochastic primitive equations.
\newblock {\em Journal of Mathematical Physics}, 55(5):051504, 2014.

\bibitem[GHT11a]{GlattHoltzTemam1}
N.~Glatt-Holtz and R.~Temam.
\newblock Cauchy convergence schemes for some nonlinear partial differential
  equations.
\newblock {\em Applicable Analysis}, 90(1):85 -- 102, 2011.

\bibitem[GHT11b]{GlattHoltzTemam2}
N.~Glatt-Holtz and R.~Temam.
\newblock Pathwise solutions of the 2-d stochastic primitive equations.
\newblock {\em Applied Mathematics and Optimization}, 63(3):401--433(33), June
  2011.

\bibitem[GHTW13]{GlattHoltzTemamWang2013}
N.~Glatt-Holtz, R.~Temam, and C.~Wang.
\newblock Martingale and pathwise solutions to the stochastic
  zakharov-kuznetsov equation with multiplicative noise.
\newblock {\em arXiv preprint arXiv:1307.6803}, 2013.

\bibitem[GHZ08]{GlattholtzZiane2008}
N.~Glatt-Holtz and M.~Ziane.
\newblock The stochastic primitive equations in two space dimensions with
  multiplicative noise.
\newblock {\em Discrete Contin. Dyn. Syst. Ser. B}, 10(4):801--822, 2008.

\bibitem[GT87]{ghidaglia1987}
Jean-Michel Ghidaglia and Roger Temam.
\newblock Attractors for damped nonlinear hyperbolic equations.
\newblock {\em Journal de math{\'e}matiques pures et appliqu{\'e}es},
  66(3):273--319, 1987.

\bibitem[Hai02]{Hairer02}
M.~Hairer.
\newblock Exponential mixing properties of stochastic {PDE}s through asymptotic
  coupling.
\newblock {\em Probab. Theory Related Fields}, 124(3):345--380, 2002.

\bibitem[HM06]{HairerMattingly06}
M.~Hairer and J.C. Mattingly.
\newblock Ergodicity of the 2{D} {N}avier-{S}tokes equations with degenerate
  stochastic forcing.
\newblock {\em Ann. of Math. (2)}, 164(3):993--1032, 2006.

\bibitem[HM08]{HairerMattingly2008}
M.~Hairer and J.C. Mattingly.
\newblock Spectral gaps in {W}asserstein distances and the 2{D} stochastic
  {N}avier-{S}tokes equations.
\newblock {\em Ann. Probab.}, 36(6):2050--2091, 2008.

\bibitem[HM11]{HairerMattingly2011}
M.~Hairer and J.C. Mattingly.
\newblock A theory of hypoellipticity and unique ergodicity for semilinear
  stochastic pdes.
\newblock {\em Electron. J. Probab.}, 16(23):658--738, 2011.

\bibitem[HMS11]{HairerMattinglyScheutzow2011}
M.~Hairer, J.~C. Mattingly, and M.~Scheutzow.
\newblock Asymptotic coupling and a general form of {H}arris' theorem with
  applications to stochastic delay equations.
\newblock {\em Probab. Theory Related Fields}, 149(1-2):223--259, 2011.

\bibitem[Kal02]{Kallenberg02}
Olav Kallenberg.
\newblock {\em Foundations of modern probability}.
\newblock Probability and its Applications (New York). Springer-Verlag, New
  York, second edition, 2002.

\bibitem[Kob07]{Kobelkov2007}
G.~M. Kobelkov.
\newblock Existence of a solution ``in the large'' for ocean dynamics
  equations.
\newblock {\em J. Math. Fluid Mech.}, 9(4):588--610, 2007.

\bibitem[KPS10]{KomorowskiPeszatSzarek2010}
T.~Komorowski, S.~Peszat, and T.~Szarek.
\newblock On ergodicity of some markov processes.
\newblock {\em The Annals of Probability}, 38(4):1401--1443, 2010.

\bibitem[KS01]{KuksinShirikyan1}
S.~Kuksin and A.~Shirikyan.
\newblock A coupling approach to randomly forced nonlinear {PDE}'s. {I}.
\newblock {\em Comm. Math. Phys.}, 221(2):351--366, 2001.

\bibitem[KS02]{KuksinShirikyan2}
S.~Kuksin and A.~Shirikyan.
\newblock Coupling approach to white-forced nonlinear {PDE}s.
\newblock {\em J. Math. Pures Appl. (9)}, 81(6):567--602, 2002.

\bibitem[KS12]{KuksinShirikian12}
S.~Kuksin and A.~Shirikyan.
\newblock {\em Mathematics of Two-Dimensional Turbulence}.
\newblock Number 194 in Cambridge Tracts in Mathematics. Cambridge University
  Press, 2012.

\bibitem[KT09]{KalantarovTiti2009}
V.~K. Kalantarov and E.~S. Titi.
\newblock Global attractors and determining modes for the 3d
  navier-stokes-voight equations.
\newblock {\em Chinese Annals of Mathematics, Series B}, 30(6):697--714, 2009.

\bibitem[KW92]{kovavcivc1992}
Gregor Kova{\v{c}}i{\v{c}} and Stephen Wiggins.
\newblock Orbits homoclinic to resonances, with an application to chaos in a
  model of the forced and damped sine-gordon equation.
\newblock {\em Physica D: Nonlinear Phenomena}, 57(1):185--225, 1992.

\bibitem[KZ07]{ZianeKukavica}
I.~Kukavica and M.~Ziane.
\newblock On the regularity of the primitive equations of the ocean.
\newblock {\em Nonlinearity}, 20(12):2739--2753, 2007.

\bibitem[Ler34]{Leray1934}
J.~Leray.
\newblock Sur le mouvement d'un liquide visqueux emplissant l'espace.
\newblock {\em Acta mathematica}, 63(1):193--248, 1934.

\bibitem[LHM78]{levi1978}
M~Levi, Frank~C Hoppensteadt, and WL~Miranker.
\newblock Dynamics of the josephson junction.
\newblock {\em Q. Appl. Math.;(United States)}, 37(3), 1978.

\bibitem[LL06]{LaytonLewandowski2006}
W.~Layton and R.~Lewandowski.
\newblock On a well-posed turbulence model.
\newblock {\em Discrete and continuous dynamical systems series B}, 6(1):111,
  2006.

\bibitem[LT09]{LariosTiti2009}
A.~Larios and E.~S. Titi.
\newblock On the higher-order global regularity of the inviscid
  voigt-regularization of three-dimensional hydrodynamic models.
\newblock {\em arXiv preprint arXiv:0910.3354}, 2009.

\bibitem[LTW92a]{LionsTemamWang2}
J.-L. Lions, R.~Temam, and S.~H. Wang.
\newblock New formulations of the primitive equations of atmosphere and
  applications.
\newblock {\em Nonlinearity}, 5(2):237--288, 1992.

\bibitem[LTW92b]{LionsTemamWang1}
J.-L. Lions, R.~Temam, and S.~H. Wang.
\newblock On the equations of the large-scale ocean.
\newblock {\em Nonlinearity}, 5(5):1007--1053, 1992.

\bibitem[LTW93]{LionsTemamWang3}
J.-L. Lions, R.~Temam, and S.~Wang.
\newblock Models for the coupled atmosphere and ocean. ({CAO} {I},{II}).
\newblock {\em Comput. Mech. Adv.}, 1(1):120, 1993.

\bibitem[Mat02]{Mattingly2002}
J.C. Mattingly.
\newblock Exponential convergence for the stochastically forced
  {N}avier-{S}tokes equations and other partially dissipative dynamics.
\newblock {\em Comm. Math. Phys.}, 230(3):421--462, 2002.

\bibitem[Mat03]{Mattingly2003}
J.C. Mattingly.
\newblock On recent progress for the stochastic {N}avier {S}tokes equations.
\newblock In {\em Journ{\'e}es ``\'{E}quations aux {D}{\'e}riv{\'e}es
  {P}artielles''}, pages Exp. No. XI, 52. Univ. Nantes, Nantes, 2003.

\bibitem[MS01]{MarsdenShkoller2001}
J.~E. Marsden and S.~Shkoller.
\newblock Global well--posedness for the lagrangian averaged navier--stokes
  (lans--$\alpha$) equations on bounded domains.
\newblock {\em Philosophical Transactions of the Royal Society of London A:
  Mathematical, Physical and Engineering Sciences}, 359(1784):1449--1468, 2001.

\bibitem[MS03]{MarsdenShkoller2003}
J.~E. Marsden and S.~Shkoller.
\newblock The anisotropic lagrangian averaged euler and navier-stokes
  equations.
\newblock {\em Archive for rational mechanics and analysis}, 166(1):27--46,
  2003.

\bibitem[MS13]{MuscaluSchlag13}
C.~Muscalu and W.~Schlag.
\newblock {\em Classical and Multilinear Harmonic Analysis}, volume 137.
\newblock Cambridge University Press, 2013.

\bibitem[Osk77]{Oskolkov1977}
A.P. Oskolkov.
\newblock The uniqueness and global solvability of boundary-value problems for
  the equations of motion for aqueous solutions of polymers.
\newblock {\em Journal of Mathematical Sciences}, 8(4):427--455, 1977.

\bibitem[Ped13]{pedlosky2013}
Joseph Pedlosky.
\newblock {\em Geophysical fluid dynamics}.
\newblock Springer Science \& Business Media, 2013.

\bibitem[PTZ08]{PetcuTemamZiane2008}
M.~Petcu, R.~Temam, and M.~Ziane.
\newblock Some mathematical problems in geophysical fluid dynamics.
\newblock In {\em Special Volume on Computational Methods for the Atmosphere
  and the Oceans}, volume~14 of {\em Handbook of Numerical Analysis}, pages
  577--750. Elsevier, 2008.

\bibitem[RY99]{RevuzYor1999}
D.~Revuz and M.~Yor.
\newblock {\em Continuous martingales and {B}rownian motion}, volume 293 of
  {\em Grundlehren der Mathematischen Wissenschaften [Fundamental Principles of
  Mathematical Sciences]}.
\newblock Springer-Verlag, Berlin, third edition, 1999.

\bibitem[Tem97]{Temam1997}
R.~Temam.
\newblock {\em Infinite-dimensional dynamical systems in mechanics and
  physics}, volume~68 of {\em Applied Mathematical Sciences}.
\newblock Springer-Verlag, New York, second edition, 1997.

\bibitem[Tem01]{Temam2001}
R.~Temam.
\newblock {\em Navier-{S}tokes equations: Theory and numerical analysis}.
\newblock AMS Chelsea Publishing, Providence, RI, 2001.
\newblock Reprint of the 1984 edition.

\bibitem[Tre92]{Trenberth1992}
Kevin~E Trenberth.
\newblock {\em Climate system modeling}.
\newblock Cambridge University Press, 1992.

\bibitem[WL02]{ELiu2002}
E.~Weinan and Di~Liu.
\newblock Gibbsian dynamics and invariant measures for stochastic dissipative
  {PDE}s.
\newblock {\em J. Statist. Phys.}, 108(5-6):1125--1156, 2002.
\newblock Dedicated to David Ruelle and Yasha Sinai on the occasion of their
  65th birthdays.

\bibitem[WZ97]{wang1997}
Guanxiang Wang and Shu Zhu.
\newblock On the dimension of the global attractor for the damped sine--gordon
  equation.
\newblock {\em Journal of Mathematical Physics}, 38(6):3137--3141, 1997.

\end{thebibliography}
\end{document}